\documentclass{article}

\usepackage{amsmath}
\usepackage{amsfonts}
\usepackage{mathrsfs}
\usepackage{amsthm}
\usepackage{makeidx}
\usepackage{graphicx}
\usepackage{amssymb}
\usepackage{wasysym}
\usepackage{latexsym}
\usepackage{dsfont}
\usepackage[left=2cm,right=2cm,top=1.5cm,bottom=1.5cm]{geometry}

\newtheorem{theorem}{\quad Theorem}
\newtheorem{definition}{\quad Definition} 
\newtheorem{corollary}{\quad Corollary} 
 
\newtheorem{lemma}{\quad Lemma}
\newtheorem{remark}{\quad Remark}
 
\newtheorem{proposition}{\quad Proposition}

\title{$H$-kernels in $H$-colored digraphs without $(\xi_{1}, \xi, \xi_{2})$-$H$-subdivisions of  $\overrightarrow{C_{3}}$}
\author{Felipe Hern\'andez-Lorenzana and Roc\'io S\'anchez-L\'opez\\
  \small Facultad de Ciencias, Universidad Nacional Aut\'onoma de M\'exico,\\
  \small C\'ircuito Exterior s/n, Coyoac\'an, Ciudad Universitaria, 04510,\\
  \small Ciudad de M\'exico, CDMX
}
\date{}

\begin{document}
\maketitle

\abstract{Let $H$ be a digraph possibly with loops  and $D$ a digraph without loops with a coloring of its arcs $c:A(D) \rightarrow V(H)$ ($D$ is said to be an $H$-colored digraph).  A directed path $W$ in $D$ is said to be an $H$-path if and only if the consecutive colors encountered on $W$ form a directed walk in $H$. A subset $N$ of vertices of $D$ is said to be an $H$-kernel  if (1) for every pair of different vertices in $N$ there is no $H$-path between them and (2) for every vertex $u$ in V($D$)$\setminus$$N$ there exists an $H$-path in $D$ from $u$ to $N$.  Under this definition an $H$-kernel is a kernel whenever  $A(H)=\emptyset$.\\
 The  color-class digraph $\mathscr{C}_C$($D$) of $D$ is the digraph whose vertices are the colors represented in the arcs of $D$ and ($i$,$j$) $\in$ $A$($\mathscr{C}_C$($D$)) if and only if there exist two arcs, namely ($u$,$v$) and ($v$,$w$) in $D$, such that ($u$,$v$) has color $i$ and ($v$,$w$) has color $j$.  Since not every $H$-colored digraph has an $H$-kernel and $V(\mathscr{C}_C(D))= V(H)$, the natural question is: what  structural properties of $\mathscr{C}_C(D)$, with respect to the $H$-coloring, imply that $D$ has an $H$-kernel?\\
In this paper we investigate the problem of the existence of an $H$-kernel by means of a partition $\xi$ of $V(H)$  and a partition \{$\xi_1$, $\xi_2$\}  of $\xi$. We establish conditions on the directed cycles  and the directed paths of the digraph $D$, with respect to the partition \{$\xi_1$, $\xi_2$\}. In particular we pay attention to some subestructures produced by the partitions $\xi$ and \{$\xi_1$, $\xi_2$\}, namely $(\xi_{1}, \xi, \xi_{2})$-$H$-subdivisions of $\overrightarrow{C_{3}}$ and $(\xi_{1}, \xi, \xi_{2})$-$H$-subdivisions of $\overrightarrow{P_{3}}$.\\ 
We  give some examples which show that each hypothesis in the main result is tight.}\\


\section{Introduction}
\label{sec:in}

For general concepts we refer the reader to \cite{2}. For a digraph $D$, $V(D)$ and $A(D)$ will denote the sets of vertices and arcs of $D$, respectively. A \emph{ directed walk} is a sequence $W$ = ($v_0$, $v_1$,  $\ldots$ , $v_n$) such that ($v_i$,$v_{i+1}$) $\in$ $A(D)$ for each $i$ in \{0,  $\ldots$ , $n-1$\}. The number $n$ is the \emph{ length} of the walk, denoted by $l(W)$. A \emph{ directed  path} is a directed walk ($v_1$,  $\ldots$ , $v_n$) such that $v_i$ $\neq$ $v_j$ for $i$ $\neq$ $j$, \{$i$, $j$\} $\subseteq$ \{0,  $\ldots$ , $n$\}. The directed path ($v_1$, $v_2$, $v_3$, $v_4$) will be denoted by $\overrightarrow{P_{3}}$. A \emph{ directed cycle} is a directed walk ($v_1$,  $\ldots$ , $v_n$, $v_1$) such that $v_i$ $\neq$ $v_j$ for $i$ $\neq$ $j$, \{$i$, $j$\} $\subseteq$ \{0,  $\ldots$ , $n$\}. The directed cycle ($v_1$, $v_2$, $v_3$, $v_1$) will be denoted by $\overrightarrow{C_{3}}$. If $D$ is an infinite digraph, an \emph{ infinite outward path} is an infinite sequence ($v_1$, $v_2$, $\ldots$ ) of different vertices of $D$ such that ($v_i$,$v_{i+1}$) $\in$ $A(D)$ for each $i$ in $\mathbb{N}$. In this paper we are going to write walk, path, cycle instead of directed walk, directed path, directed cycle, respectively. A digraph $D$ is \emph{ acyclic} if it has no   cycle.  The union of walks will be denoted with $\cup$.  Let $W$ = ($v_0$, $v_1$, $\ldots$ , $v_n$) be a walk and \{$v_i$,$v_j$\} $\subseteq$ $V(W)$, with $i \textless j$. The $v_i$$v_j$-walk ($v_i$,$v_{i+1}$, $\ldots$ , $v_{j-1}$,$v_j$) contained in $W$ will be denoted by ($v_i$,$W$,$v_j$).

For an arc ($u$,$v$) the first vertex $u$ is its \emph{ tail} and the second vertex $v$ is its \emph{ head}. For a vertex $v$, the \emph{ out-degree} is the number of arcs, except for loops, with tail $v$, which is denoted by $d^{+}(v)$.

For $S \subseteq  V(D)$  the \emph{subdigraph of D induced  by   S}, denoted by $D$[$S$], is the digraph such that $V(D[S]) = S$ and $A(D[S]) = \{(u,v)\in A(D) :  \{u, v\} \subseteq S\}$.  For $A \subseteq  A(D)$  the \emph{ subdigraph of $D$ induced  by  $A$}, denoted by $D[A]$,  is the digraph such that $V(D[A]) = \{v : v$ is either the tail or the head of an arc $a$ for some $a$ in $A$\}  and its arc set is $A$. We shall say that a subset $S$ of V($D$) is \emph{ independent} if  $D[S]$ has no arcs. A digraph $D$ is a \emph{ bipartite digraph} if there exists a partition \{$V_1$, $V_2$\} of $V(D)$ such that $D[V_i]$ is an independent set for each $i$ in \{1, 2\}. Let $R_1$ = \{$S$ : $S$ is a subset of $V(D)$\} and $R_2$ = \{$T$ : $T$ is a subset of $V(D)$\} two family sets; an $R_1R_2$\emph{-arc} is an arc ($u$, $v$) of $D$ such that $u\in S$ for some $S$ in $R_1$ and $v\in T$ for some $T$ in $R_2$.  A pair of digraphs $D$ and $G$ are \emph{ isomorphic} if there exists a bijection $f : V(D) \rightarrow V(G)$ such that ($x$,$y$) $\in$ $A(D)$ if and only if $(f(x),f(y)) \in A(G)$ ($f$ will be called isomorphism). We will say that a digraph $D$ is \emph{ complete} if for every pair of different vertices $u$ and $v$ in $V(D)$ it holds that $\{(u,v), (v,u)\} \subseteq A(D)$. A digraph $D$ is \emph{ transitive} whenever $\{(u,v), (v,w)\} \subseteq A(D)$ implies that $(u,w) \in A(D)$.

A digraph $D$ is said to be \emph{m-colored} if the arcs of $D$ are colored with $m$ colors.   A \emph{ chromatic class} of $D$ is the set of arcs of a same color. We say that a chromatic class $\cal{C} $ is transitive if $D[\cal C]$ is a transitive digraph. A  path  is called \emph{ monochromatic} if all of its arcs are colored alike. A path $P$ is called \emph{ properly colored} if consecutive arcs in $P$ have different color.  A path  is called \emph{ rainbow} if all arcs have distinct colors.  For an arc ($z_1$,$z_2$) of $D$ we will denote by $c(z_1,z_2)$ its color. \\

Reachability is a topic widely studyed due to its applications. A number of variants of this concept have appeared in the last decades, for example, reachability by monochromatic paths, by rainbow paths, or by properly colored paths. For every notion of reachability, notions of independence and absorbency can be defined, and thus, a notion of kernel too.

In \cite{23} Sands, Sauer and Woodrow prove that every digraph whose arcs are colored with two colors, which has no monochromatic infinite outward path,  contains a set $S$ of vertices such that, no two vertices of $S$ are connected by a monochromatic directed path, and for every vertex $x$ not in $S$ there is a monochromatic directed path from $x$ to a vertex in $S$. In \cite{4}  Linek and Sands generalize  the notion of  monochromatic path in the following way: let $H$ be a digraph, possibly with loops,  and $D$ a digraph without loops; an  $H$\emph{-arc coloring} of $D$ is a function $c: A(D) \rightarrow V(H)$. $D$ is said to be $H$\emph{-colored}  if $D$ has an $H$-arc coloring. A  path $W = (v_0 , \ldots , v_n)$ in $D$ is said to be an $H$\emph{-path} if and only if $(c(v_0,v_1), \ldots , c(v_{n-1},v_n))$ is a  walk in $H$. In \cite{1} Arpin and Linek show an extension of the concept of $H$-path and they define $H$\emph{-walk} as a walk $W = (v_0 , \ldots , v_n)$ such that $(c(v_0,v_1), \ldots ,$ $c(v_{n-1},v_n))$ is a  walk in $H$. We  consider that an arc is an $H$-walk, that  is, a singleton vertex is a walk in $H$.  In \cite{1} Arpin and Linek, in particular,  make inroads in the classification of $\mathscr{B}_3$ (the class of all $H$ such that any multidigraph $D$ arc-colored with the vertices of $H$ has a set $S$ of vertices such that (1) there is no $H$-walk between any two distinct vertices of $S$ ($S$ is \emph{$H$-independent by walks}) and (2) for every $x$ in $V(D)\setminus S$ there is an $H$-walk from $x$ to some point of $S$ ($S$ is \emph{$H$-absorbent by walks}). Let $S$ be a subset of $V(D)$, $S$ is said to be \emph{$H$-kernel by walks} if $S$ is both  $H$-absorbent by walks  and  $H$-independent by walks.
 
 Since not every $uv$-$H$-walk contains a $uv$-$H$-path and the concatenation of two $H$-walks is an $H$-walk, in \cite{33} Galeana-S\'anchez and  Delgado-Escalante introduce the concept of \emph{$H$-kernel} in a digraph $D$ as a subset  $S$ of $V(D)$ which is both \emph{ $H$-absorbent} (for every $x$ in $V(D)\setminus S$ there is an $H$-path from $x$ to some point of $S$) and \emph{$H$-independent}  (there is no $H$-path between any two distinct vertices of $S$).

It follows from the definition of $H$-kernel that: when $A(H)=\emptyset$,  an $H$-kernel is a \emph{ kernel}; when $A(H)=\{(u,u) : u \in V(H)\}$,  an $H$-kernel is a \emph{ kernel by monochromatic paths} (mp-kernel); when $H$ has no loops,  an $H$-kernel is a \emph{ kernel by properly colored paths} (PCP-kernel) and when $H$ has no cycles,  an $H$-kernel is a \emph{ kernel by rainbow paths}. In each of these special cases for $H$, sufficient conditions have been established in order to guarantee the existence of $H$-kernels, see for example \cite{105}, \cite{104},  \cite{22}, \cite{5},   \cite{23}. 

In \cite{6} we find that it is NP-complete to recognize whether a digraph has a kernel. In \cite{106} the authors prove that  the problem of determining whether an $H$-colored digraph has a kernel by $H$-walks is in NP. In \cite{105} we find that (1) it is NP-hard to recognize whether an arc-colored digraph has a PCP-kernel and (2) it is NP-hard to recognize whether an arc-colored digraph has a kernel by rainbow  paths. Due to the difficulty of finding kernels, mp-kernels, alternating kernels and kernels by rainbow paths in arc-colored digraphs, sufficient conditions for the existence of each of these $H$-kernels in arc-colored digraphs have been obtained mainly by study special classes of digraphs.

An interesting digraph associated with an arc-colored digraph $D$ is the \emph{ color-class digraph}, $\mathscr{C}_C(D)$,  which is defined as the digraph whose vertices are the colors represented in the arcs of $D$ and $(i,j) \in A(\mathscr{C}_C(D))$ if and only if there exist two arcs, namely ($u$,$v$) and ($v$,$w$) in $D$, such that ($u$,$v$) has color $i$ and ($v$,$w$) has color $j$ (notice that $\mathscr{C}_C(D)$ can have loops by  definition).  With this associated digraph, Galeana-S\'anchez obtained an extension of Sands, Sauer and Woodrow's theorem  in \cite{100}, for the finite case, as follows.

\begin{theorem}\cite{100}
\label{ccd}
Let $D$ be a finite $m$-colored digraph. If $\mathscr{C}_C(D)$ is a bipartite digraph, then $D$ has an mp-kernel.
\end{theorem}

Notice that in \cite{100}  Galeana-S\'anchez work with a partition of $V(\mathscr{C}_C(D))$ into two independent sets. Since an $H$-kernel is an mp-kernel  whener $A(H)=\{(u,u) : u \in V(H)\}$, we can say that Galeana-S\'anchez work with a partition of $V(H)$ which holds a property with respect to the  digraph $\mathscr{C}_C(D)$. In \cite{galeana}, \cite{gammagaleana}  and \cite{gaytan}  we also find that the authors work with a partitions of $V(H)$ in order to guarantee the existence of mp-kernels.\\

The  results in \cite{galeana}, \cite{gammagaleana}, \cite{100}  and \cite{gaytan}  motivate us to continue with the study of the existence of $H$-kernels by means of  a partition of $V(H)$. In this paper we  work with a partition of $V(H)$, say $\xi$, and with a partition \{$\xi_1$, $\xi_2$\} of $\xi$ which satisfy certain properties. In order to show those properties, we need more definitions.\\

Let $H$ be a digraph and $D$ an $H$-colored digraph. We will say that $D$ is \emph{ transitive by H-paths} if the existence of an $xy$-$H$-path and the existence of a $yz$-$H$-path in $D$ imply that there exists an $xz$-$H$-path in $D$. Let  ($v_0$, $v_1$, $\ldots$ , $v_n$) be a walk in $D$; we  say that there is an \emph{H-obstruction  on  $v_i$} ~ if ($c$($v_{i-1}$,$v_i$),$c$($v_i$,$v_{i+1}$)) $\notin$ A($H$) (if $v_0$ = $v_n$ we  take indices modulo $n$). Let  $\xi$ be a partition of $V(H)$ and  \{$\xi_1$, $\xi_2$\} a partition of $\xi$. $D_i$ is the spanning subdigraph of $D$ such that $A(D_i)=\{a\in A(D) : c(a)\in C ~\text{for some} ~C ~ \text{in} ~ \xi_i\}$ for every $i$ in \{1, 2\} .

 Let $W=(u_{0}, \ldots , u_{l}=v_{0}, \ldots , v_{m}=w_{0}, \ldots ,w_{n}=u_{0})$ be a cycle, we  say that $W$ is a $(\xi_{1}, \xi, \xi_{2})$\emph{-H-subdivision of} $\overrightarrow{C_{3}}$ if $T_{1}=(u_{0},W, u_{l})$ is an $H$-path contained in $D_{1}$, $T_{2}=(v_{0},W, v_{m})$ is an $H$-path contained in $D$ and $T_{3}=(w_{0}, W, w_{n})$ is an $H$-path contained in $D_{2}$, where there are $H$-obstructions on $u_{0}$, $v_{0}$ and $w_{0}$ with respect to $W$. 

Let $P=(u_{0}, \ldots , u_{l}=v_{0}, \ldots , v_{m}=w_{0}, \ldots ,w_{n})$ be a path, we  say that $P$ is a $(\xi_{1}, \xi, \xi_{2})$\emph{-H-subdivision   of} $\overrightarrow{P_{3}}$ if $T_{1}=(u_{0},P, u_{l})$ is an $H$-path contained in $D_{1}$, $T_{2}=(v_{0}, P, v_{m})$ is an $H$-path contained in $D$ and $T_{3}=(w_{0}, P, w_{n})$ is an $H$-path contained in $D_{2}$, where there are $H$-obstructions on  $v_{0}$ and $w_{0}$ with respect to $P$.\\

The main result is the following:\\

Let $H$ be a digraph possibly with loops and $D$ an $H$-colored digraph without isolated vertices. Let $\xi$ =\{$C_{1}$, $C_{2}$, $\ldots$ ,  $C_{k}$\} ($k\geq 2$) be  a partition of $V(H)$ such that every $i$ in $\{1, 2, \ldots , k\}$  holds that   $G_{i}=D[\{a\in A(D) : c(a)\in C_{i}\}]$ is a subdigraph of $D$ which is transitive by $H$-paths in $D$, suppose that $\{a\in A(D) : c(a)\in C_{i}\} \neq \emptyset$. Let \{$\xi_1$, $\xi_2$\} be a partition of $\xi$.  Suppose that
	\begin{enumerate}
		\item for every $i$ in $\{1, 2\}$ and for every cycle $\gamma$ contained in $D_{i}$ there exists $C_m$  in $\xi_{i}$ such that $\gamma$ is contained in $G_{m}$,
		\item for every $i$ in $\{1, 2\}$ and for every $H$-walk $P$ contained in $D_{i}$ there exists $C_{m'}$  in $\xi_{i}$ such that $P$ is contained in $G_{m'}$,
         \item if either there exists a $\xi_{1}\xi_{2}$-arc or there exists a $\xi_{2}\xi_{1}$-arc in $A(\mathscr{C}_{C}(D))$, say $(a,b)$, then $(a,b) \notin A(H)$,
         \item $D$ does not contain a $(\xi_{1}, \xi, \xi_{2})$-$H$-subdivision of $\overrightarrow{C_{3}}$,
		\item if there exists a $ux$-path which is a $(\xi_{1}, \xi, \xi_{2})$-$H$-subdivision of $\overrightarrow{P_{3}}$, for some subset $\{u,x\}$ of $V(D)$, then there exists a $ux$-$H$-path in $D$.
	\end{enumerate}  	

Then $D$ has an $H$-kernel.\\

We will see that  Theorem \ref{ccd} and the main result in \cite{galeana} are direct consequences of the  main result of this paper and we deduce some results that show the existence of  kernels by properly colored paths and kernels by rainbow paths. We  finish with some examples which show that each hypothesis in the main result is tight.\\

We  need the following results.\\

\begin{proposition}[\cite{2}]
\label{prop100}
Every acyclic digraph has a vertex $v$ such that $d^{+}(v)=0$.
\end{proposition}

\begin{proposition}
\label{closedcycle}
Every closed walk contains a cycle.
\end{proposition}

If every induced subdigraph of $D$ has a kernel, $D$ is said to be a \emph{ kernel perfect digraph}.  Among the classical results on the theory of kernels we have the following theorem.

\begin{theorem}[\cite{5}]
\label{0000}
Every acyclic digraph has a kernel.
\end{theorem}

\section{Previous results}
\label{sec:first}

Since every isolated vertex in $D$ is in every $H$-kernel of $D$,  in this paper we  suppose that $D$ has no isolated vertices.

From now on $H$ is a finite digraph possibly with loops, $D$ is a finite $H$-colored digraph without loops and $\xi$ =\{$C_{1}$, $C_{2}$, $\ldots$ ,  $C_{k}$\} ($k\geq 2$)  is  a partition of $V(H)$ such that for every $i$ in $\{1, 2, \ldots , k\}$ we have that $\{a\in A(D) : c(a)\in C_{i}\} \neq \emptyset$ and $G_{i}=D[\{a\in A(D) : c(a)\in C_{i}\}]$ is a subdigraph of $D$ which is transitive by $H$-paths in $D$. Notice that \{$\{a\in A(D):c(a)\in C_{i}\} : i \in \{1, 2, \ldots , k\}$\} is a partition of $A(D)$.

\begin{lemma}
\label{lema:indtray}
Let $H$ be a digraph, $D$ an $H$-colored digraph, $l$ in $\{1,2, \ldots ,k\}$ and $(x_{0}$, $x_{1}$, $\ldots$ , $x_{n-1})$ a sequence of $n \geq 2$ vertices, different by pairs. If for each $i$ in $\{1,2, \ldots , n-1\}$ we have that there exists an $x_{i-1}x_{i}$-$H$-path in $G_{l}$, then for each $m$ in $\{1,2, \ldots , n-1\}$ there exists an $x_{0}x_{m}$-$H$-path contained in $G_{l}$.
\end{lemma}
\begin{proof}
	We proceed by induction on $n$. 
	
If $n=2$, then for the sequence $(x_{0},x_{1})$  we have by hypothesis that there exists an $x_{0}x_{1}$-$H$-path contained in $G_{l}$.
	
Suppose that if $(y_{0},y_{1}, \ldots , y_{m-1})$ is a sequence of $m$ vertices, with $2 \leq m<n$, that satisfies the hypothesis of Lemma~\ref{lema:indtray}, then for every $m'$ in $\{1,2, \ldots , m-1\}$ there exists a $y_{0}y_{m'}$-$H$-path contained in $G_{l}$.
	
Let $(x_{0},x_{1}, \ldots , x_{n-1})$ be a sequence of $n$ vertices, with $n\geq 3$, which satisfies the hypothesis of Lemma~\ref{lema:indtray}. It follows from the induction hypothesis on the sequence $(x_{0},x_{1}, \ldots , x_{n-2})$ that for every $m'$ in $\{1,2, \ldots , n-2\}$ there exists an $x_{0}x_{m'}$-$H$-path $T_{m'}$ which is contained in $G_{l}$. On the other hand, let $T$ be an $x_{n-2}x_{n-1}$-$H$-path which is contained in $G_{l}$ ($T$ there exists by hypothesis) and let $T_{n-2}$ be the $x_{0}x_{n-2}$-$H$-path  which is also contained in $G_{l}$. Since $G_{l}$ is transitive by $H$-paths, we get that there exists an $x_{0}x_{n-1}$-$H$-path in $G_{l}$. Therefore, for every $m'$ in $\{1,2, \ldots , n-1\}$ there exists an $x_{0}x_{m'}$-$H$-path contained in $G_{l}$.
\end{proof}

\begin{definition}
\label{defi:Hseminucleo}
	Let $H$ be a digraph, $D$ an $H$-colored digraph and $S$ a subset of $V(D)$. We   say that $S$ is an H-semikernel if $S$ satisfies the following:
	\begin{enumerate}
		\item $S$ is an $H$-independent set in $D$.
		\item For every $z$ in $V(D)\setminus S$, if there exists a $Sz$-$H$-path in $D$, then there exists a $zS$-$H$-path in D.
	\end{enumerate}
\end{definition}

\begin{lemma}\label{lema:existenciaHseminucleogr}
Let $H$ be a digraph, $D$ an $H$-colored digraph and $r$ in $\{1,2, \ldots ,k\}$. Then 

\begin{enumerate}
\item There exists no a sequence of vertices $(x_{0}, x_{1}, x_{2}, \ldots)$ such that for every $i$ in $\{0,1,2, \ldots\}$ there exists an $x_{i}x_{i+1}$-$H$-path in $G_{r}$ and there exists no an $x_{i+1}x_{i}$-$H$-path in $G_{r}$.
\item There exists $x_{0}$ in $V(G_{r})$ such that $\{x_{0}\}$ is an $H$-semikernel of $G_{r}$.
\end{enumerate}
\end{lemma}
\begin{proof}

\begin{enumerate}
\item Proceeding by contradiction, suppose that there exists a sequence of vertices $(x_{0}, x_{1}, \ldots)$ such that for every $i$ in $\{0, 1, 2, \ldots\}$ there exists an $x_{i}x_{i+1}$-$H$-path in $G_{r}$ and there exists no an $x_{i+1}x_{i}$-$H$-path in $G_{r}$.

Since $D$ is a finite digraph, we get that there exists a subset $\{i, j\}$ of $\{0, 1, 2, \ldots\}$, with $i<j$, such that $x_{i}=x_{j}$. Let $j_{0}=$min$\{j\in \mathbb{N} : x_{j}=x_{i}$ for some $i<j\}$ be and $i_{0}$ in $\{0, 1, \ldots , j_{0}-1\}$ such that $x_{i_{0}}=x_{j_{0}}$.  Notice that it follows from the choice of $j_{0}$ that $(x_{i_0}, x_{i_0 +1}, \ldots , x_{j_0 -1})$ is a sequence of vertices, different by pairs. Suppose without loss of generality that $i_{0}=0$, $j_{0}=n$, then $(x_{0}, x_{1}, x_{2}, \ldots , x_{n-1})$ is a sequence of $n \geq 2$ vertices, different by pairs, such that for each $i$ in $\{0, 1, \ldots , n-1\}$ we have that there exists an $x_{i}x_{i+1}$-$H$-path in $G_{r}$ and there exists no an $x_{i+1}x_{i}$-$H$-path in $G_{r}$. For each $i$ in $\{0, 1, \ldots , n-1\}$ let $T_{i}$ be an $x_{i}x_{i+1}$-$H$-path in $G_{r}$ (indices modulo $n$), then we get from  Lemma~\ref{lema:indtray} that there exists an $x_{0}x_{n-1}$-$H$-path contained in $G_{r}$; that is, there exists an $x_{n}x_{n-1}$-$H$-path contained in $G_{r}$ which is not possible.

\item Proceeding by contradiction, suppose that for each $x$ in $V(G_{r})$, $\{x\}$ is not an $H$-semikernel of $G_{r}$; that is, for every $x$ in $V(G_{r})$ there exists $y$ in $V(G_{r})\setminus \{x\}$ such that there exists an $xy$-$H$-path contained in $G_{r}$ and there exists no a $yx$-$H$-path in $G_{r}$. Therefore, for every $n$ in $\mathbb{N}$ we have that given $x_n$ in $V(G_{r})$ there exists $x_{n+1}$ in $V(G_{r})\setminus \{x_{n}\}$ such that there exists an $x_{n}x_{n+1}$-$H$-path in $G_{r}$ and there exists no an $x_{n+1}x_{n}$-$H$-path in $G_{r}$ which implies that $(x_{0}, x_{1}, x_{2}, \ldots)$ is a sequence of vertices  that contradicts (1).
\end{enumerate}
\end{proof}

\begin{lemma}\label{lema:unionhtray}
	Let $H$ be a digraph and $D$ an $H$-colored digraph. Suppose that
	\begin{enumerate}
		\item For every cycle $\gamma$ in $D$ there exists $i$  in $\{ 1, 2, \ldots , k\}$ such that $\gamma$ is contained in $G_{i}$.
		\item For every $H$-walk $P$ in $D$ there exists $j$  in $\{ 1, 2, \ldots , k\}$ such that $P$ is contained in $G_{j}$.
	\end{enumerate} 
	
	If $S=(u_{0}, u_{1}, \ldots , u_{n-1})$ is a sequence of $n \geq 2$ vertices, different by pairs, such that for every $i$ in $\{0, 1, \ldots , n-1\}$ there exists a $u_{i}u_{i+1}$-$H$-path in $D$, say $T_{i}$, then there exists $j$  in $\{ 1, 2, \ldots , k\}$ such that $\bigcup \limits _{i=0} ^{n-1} T_{i}$ is contained in $G_{j}$ (indices modulo $n$).
\end{lemma}
\begin{proof}  We proceed by induction on $n$, the number of vertices in $S$.
	
If $n=2$, then $T_{0}\cup T_{1}$ is a closed walk which contains a cycle $\gamma$ (by Proposition \ref{closedcycle}). It follows from hypothesis (1) that there exists $j$  in $\{ 1, 2, \ldots , k\}$ such that $\gamma$ is contained in $G_{j}$. On the other hand, let $\{i_{0}, i_{1}\}$ be a subset of $\{ 1, 2, \ldots , k\}$ such that $T_{0}$ is contained in $G_{i_{0}}$ and $T_{1}$ is contained in $G_{i_{1}}$ (by hypothesis (2)). Since $\gamma$ contains arcs of $T_{0}$ and $T_{1}$ we get that  $i_{0}=j=i_{1}$ which implies that $T_{0}\cup T_{1}$ is contained in $G_{j}$.
	
Suppose that if $S'=(u_{0}, u_{1}, \ldots , u_{l-1})$ is a sequence of $l$ vertices, different by pairs,  with $n-1\geq l\geq2$, such that for every $i$ in $\{0, 1, \ldots , n-1\}$ there exists a $u_{i}u_{i+1}$-$H$-path in $D$, say $T'_{i}$, then there exists $j$  in $\{ 1, 2, \ldots , k\}$ such that $\bigcup \limits _{i=0} ^{l-1} T'_{i}$ is contained in $G_{j}$. 
	
Let $S=(u_{0}, u_{1}, \ldots , u_{n-1})$ be a sequence of $n$ vertices which satisfies the hypothesis of Lema \ref{lema:unionhtray}. From hypothesis (2) we get that for every $i$ in $\{0, 1, \ldots , n-1\}$ there exists $i'$  in $\{ 1, 2, \ldots , k\}$ such that $T_{i}$ is contained in $ G_{i'}$.
	
	Consider two cases:

	\it{Case 1.}  There exists $i$ in $\{0, 1, \ldots , n-1\}$ such that $u_{i}$ is not an $H$-obstruction in $\bigcup \limits _{i=0}^{n-1} T_{i}$.
	
	Suppose without loss of generality that $u_{1}$ is not an $H$-obstruction in $\bigcup \limits _{i=0}^{n-1} T_{i}$. Then $T_{0}\cup T_{1}$ is a $u_{0}$$u_{2}$-$H$-walk in $D$ which implies that there exists $m$  in $\{ 1, 2, \ldots , k\}$  such that $T_{0}\cup T_{1}$ is contained in $G_{m}$ (by hypothesis (2)).  Since $G_{m}$ is transitive by $H$-paths and both $T_{0}$ and $T_{1}$ are contained in $G_{m}$ it follows that there exists a $u_{0}$$u_{2}$-$H$-path $T$ in $G_{m}$. Consider the sequence $S'=(u_{0}, u_{2}, u_{3}, \ldots , u_{n-1})$, which is a sequence that satisfies the hypotheses of Lemma ~\ref{lema:unionhtray} and so it follows from the  induction hypothesis  that there exists $l$  in $\{ 1, 2, \ldots , k\}$ such that $T\cup T_{2}\cup T_{3}\cup \ldots \cup T_{n-1}$ is contained in $G_{l}$ and since $T$ is contained in $G_{m}$ we have that $m=l$. Therefore, $\bigcup \limits _{i=0}^{n-1} T_{i}$ is contained in $G_{m}$.

	\it{Case 2.} For every $i$ in $\{0, 1, \ldots , n-1\}$ we have that $u_{i}$ is an $H$-obstruction in $\bigcup \limits _{j=0}^{n-1}T_{j}$.
	
	Consider two subcases.

\it{Case 2.1.}	 $(V(T_{i})-\{u_{i+1}\})\cap V(T_{i+1})\neq \emptyset$ for some $i$ in $\{0, 1, \ldots , n-1\}$. 

Let $v$ be a vertex in $(V(T_{i})-\{u_{i+1}\})\cap V(T_{i+1})$, then $(v, T_{i}, u_{i+1})\cup (u_{i+1}, T_{i+1}, v)$ is a closed walk which contains a cycle $\gamma$ ( by Proposition \ref{closedcycle}). It follows from hypothesis (1) that there exists $m$  in $\{ 1, 2, \ldots , k\}$ such that $\gamma$ is contained in $G_{m}$ and since $\gamma$ contains arcs of both $T_{i}$ and $T_{i+1}$ it follows that $T_{i}$ and $T_{i+1}$ are contained in $G_{m}$; that is, $i'=m=(i+1)'$.  Because  of that $G_{m}$ is transitive by $H$-paths it follows that there exists a $u_{i}u_{i+2}$-$H$-path $T$ contained in $G_{m}$. Therefore, since $S'=(u_{0}, u_{1}, \ldots , u_{i}, u_{i+2}, \ldots , u_{n-1})$ is a sequence which holds the hypotheses of Lemma~\ref{lema:unionhtray} we get from induction hypothesis that there exists $l$  in $\{ 1, 2, \ldots , k\}$ such that $T_{0}\cup T_{1}\cup \ldots \cup T_{i-1} \cup T \cup T_{i+2}\cup \ldots \cup T_{n-1}$ is contained in $G_{l}$. Since $T$ is contained in $G_{m}$ we get that $m=l$. Therefore, $\bigcup \limits _{i=0}^{n-1} T_{i}$ is contained in $G_{m}$.

\it{Case 2.2.}  $(V(T_{i})-\{u_{i+1}\})\cap V(T_{i+1})=\emptyset$ for every $i$ in $\{0, 1, \ldots , n-1\}$.

If $V(T_{i})\cap V(T_{j})=\emptyset$ for every subset $\{i, j\}$ of $\{1, 2, \ldots , k\}$ such that $|i-j|\geq 2$, then $\bigcup \limits _{i=0}^{n-1} T_{i}$ is a cycle and by hypothesis (1) we get that there exists $m$  in $\{ 1, 2, \ldots , k\}$ such that $\bigcup \limits _{i=0}^{n-1} T_{i}$ is contained in $G_{m}$.
			
Suppose that $V(T_{i})\cap V(T_{j})\neq \emptyset$ for some subset \{$i$, $j$\} of $\{1, 2, \ldots , k\}$ with $|i-j|\geq 2$. Assume without loss of generality that $i<j$ and let $v$ be a vertex in $V(T_{i})\cap V(T_{j})$.

 If $v=u_{i}$, then consider the $H$-paths $T_{j}'=(u_{j}, T_{j}, v=u_{i})$ and $T_{j}''=(v=u_{i}, T_{j}, u_{j+1})$ such that $T_{j}'\cup T_{j}''=T_{j}$. Since the sequences $S'=(u_{i}=v, u_{j+1}, u_{j+2}, \ldots , u_{n-1}, u_{0}, u_{1}, \ldots , u_{i-1})$ and $S''=(u_{i}=v, u_{i+1}, u_{i+2}, \ldots , u_{j-1}, u_{j})$ hold the induction hypothesis we get that there exist $r$ and $s$ in $\{ 1, 2, \ldots , k\}$ such that $T_{j}''\cup( \bigcup \limits _{m=j+1}^{n-1} T_{m})\cup (\bigcup \limits _{m=0}^{i-1} T_{m})$ is contained in $G_{r}$ and $(\bigcup \limits _{m=i}^{j-1}T_{m})\cup T_{j}'$ is contained in $G_{s}$. Since \{$\{a\in A(D):c(a)\in C_{i}\} : i \in \{1, 2, \ldots , k\}$\} is a partition of $A(D)$ it follows that $r=j'=s$ (recall that $T_{j}$ is contained in $G_{j'}$). Therefore, $\bigcup \limits _{m=0}^{n-1} T_{m}$ is contained in $G_{s}$.

If $v=u_{i+1}$, then consider the $H$-paths $T_{j}'=(u_{j}, T_{j}, v=u_{i+1})$ and $T_{j}''=(v=u_{i+1}, T_{j}, u_{j+1})$ such that $T_{j}'\cup T_{j}''=T_{j}$. It follows from the induction hypothesis on the two sequences  $S'=(u_{i+1}=v, u_{j+1}, u_{j+2}, \ldots , u_{n-1}, u_{0}, u_{1}, \ldots , u_{i})$ and $S''=(u_{i+1}=v, u_{i+2}, \ldots , u_{j-1}, u_{j})$ that there exist $r$ and $s$ in $\{1, 2, \ldots , k\}$ such that  $T_{j}''\cup( \bigcup \limits _{m=j+1}^{n-1} T_{m})\cup (\bigcup \limits _{m=0}^{i} T_{m})$ is contained in $G_{r}$ and $(\bigcup \limits _{m=i+1}^{j-1}T_{m})\cup T_{j}'$ is contained in $G_{s}$. Notice that $r=j'=s$ which implies that $\bigcup \limits _{m=0}^{n-1}T_{m}$ is contained in $G_{s}$.

If $v=u_{j}$, then we can consider the $H$-paths $T_{i}'=(u_{i}, T_{i}, v=u_{j})$ and $T_{i}''=(v=u_{j}, T_{i}, u_{i+1})$  (notice that $T_{i}'\cup T_{i}''=T_{i}$). Since the sequences $S'=(u_{j}=v, u_{i+1}, u_{i+2}, \ldots , u_{j-1})$ and $S''=(u_{j}=v, u_{j+1}, \ldots , u_{n-1}, u_{0}, u_{1}, \ldots , u_{i})$ hold the induction hypothesis we get that there exist $r$ and $s$ in $\{1, 2, \ldots , k\}$ such that $T_{i}''\cup( \bigcup \limits _{m=i+1}^{j-1} T_{m})$ is contained in $G_{r}$ and $(\bigcup \limits _{m=j}^{n-1}T_{m})\cup (\bigcup \limits _{m=0}^{i-1} T_{m})\cup T_{i}'$ is contained in $G_{s}$. Because of that $r=i'=s$ it follows that $\bigcup \limits _{m=0}^{n-1}T_{m}$ is contained in $G_{s}$.

If $v=u_{j+1}$, then consider the $H$-paths $T_{i}'=(u_{i}, T_{i}, v=u_{j+1})$ and $T_{i}''=(v=u_{j+1}, T_{i}, u_{i+1})$ such that $T_{i}'\cup T_{i}''=T_{i}$. Since the sequences $S'=(u_{j+1}=v, u_{i+1}, u_{i+2}, \ldots , u_{j})$ and $S''=(u_{j+1}=v, u_{j+2}, \ldots , u_{n-1}, u_{0}, u_{1}, \ldots , u_{i})$ hold the induction hypothesis it follows that there exist $r$ and $s$ in $\{1, 2, \ldots , k\}$ such that $T_{i}''\cup( \bigcup \limits _{m=i+1}^{j} T_{m})$ is contained in $G_{r}$ and $(\bigcup \limits _{m=j+1}^{n-1}T_{m})\cup (\bigcup \limits _{m=0}^{i-1} T_{m})\cup T_{i}'$ is contained in $G_{s}$. Because of that $r=i'=s$ we get that $\bigcup \limits _{m=0}^{n-1}T_{m}$ is contained in $G_{s}$.

Suppose that $v\notin \{u_{i}, u_{i+1}, u_{j}, u_{j+1}\}$, then consider the $H$-paths $T_{i}'=(u_{i}, T_{i}, v)$, $T_{i}''=(v, T_{i}, u_{i+1})$, $T_{j}'=(u_{j}, T_{j}, v)$, $T_{j}''=(v, T_{j}, u_{j+1})$ such that $T_{i}'\cup T_{i}''=T_{i}$ and $T_{j}'\cup T_{j}''=T_{j}$. Since the sequences $S'=(v, u_{i+1}, u_{i+2}, \ldots , u_{j})$ and $S''$ = $(v$, $u_{j+1}$, $u_{j+2}$, $\ldots$ , $u_{n-1}$, $u_{0}$, $u_{1}$, $\ldots$ , $u_{i})$ hold the induction hypothesis we get that there exist $r$ and $s$ in $\{1, 2, \ldots , k\}$ such that $T_{i}''\cup( \bigcup \limits _{m=i+1}^{j-1} T_{m})\cup T_{j}'$ is contained in $G_{r}$ and $T_{j}''\cup (\bigcup \limits _{m=j+1}^{n-1}T_{m})\cup (\bigcup \limits _{m=0}^{i-1} T_{m})\cup T_{i}'$ is contained in $G_{s}$. Since \{$\{a\in A(D):c(a)\in C_{i}\} : i \in \{1, 2, \ldots , k\}$\} is a partition of $A(D)$ it follows that $r=i'=j'=s$ (recall that $T_{j}$ is contained in $G_{j'}$ and $T_{i}$ is contained in $G_{i'}$) which implies that  $\bigcup \limits _{m=0}^{n-1}T_{m}$ is contained in $G_{s}$.

Therefore, we conclude from Cases 1 and 2 that there exists $s$ in $\{1, 2, \ldots , k\}$ such that $\bigcup \limits _{i=0}^{n-1}T_{i}$ is contained in $G_{s}$. 
\end{proof}

\begin{lemma}\label{lema:noexistenciasucesion}
	Let $H$ be a digraph and $D$ an $H$-colored digraph. Suppose that
	\begin{enumerate}
		\item For every cycle $\gamma$ in $D$ there exists $i$  in $\{ 1, 2, \ldots , k\}$ such that $\gamma$ is contained in $G_{i}$.
		\item For every $H$-walk $P$ in $D$ there exists $j$  in $\{ 1, 2, \ldots , k\}$ such that $P$ is contained in $G_{j}$.
	\end{enumerate} 
	
	Then there exists no a sequence of vertices $(x_{0}, x_{1}, x_{2}, \ldots)$ such that for every $i$ in $\{0, 1, 2, \ldots \}$ there exists an $x_{i}x_{i+1}$-$H$-path in $D$ and there  exists no an $x_{i+1}x_{i}$-$H$-path in $D$.
\end{lemma}
\begin{proof} Proceeding by contradiction, suppose that there exists a sequence of vertices $(x_{0}, x_{1}, \ldots)$ such that for every $i$ in $\{0, 1, 2, \ldots\}$ there exists an $x_{i}x_{i+1}$-$H$-path in $D$ and there exists no an $x_{i+1}x_{i}$-$H$-path in $D$.
	
Since $D$ is a finite digraph, we get that there exists a subset $\{i, j\}$ of $\{0, 1, 2, \ldots\}$, with $i<j$, such that $x_{i}=x_{j}$. Let $j_{0}=$min$\{j\in \mathbb{N} : x_{j}=x_{i}$ for some $i<j\}$ be and $i_{0}$ in $\{0, 1, \ldots , j_{0}-1\}$ such that $x_{i_{0}}=x_{j_{0}}$.   Suppose without loss of generality that $i_{0}=0$, $j_{0}=n$, then $(x_{0}, x_{1}, x_{2}, \ldots , x_{n-1})$ is a sequence of $n \geq 2$ vertices, different by pairs, such that for each $i$ in $\{0, 1,  \ldots , n-1\}$ we have that there exists an $x_{i}x_{i+1}$-$H$-path in $D$ and there exists no an $x_{i+1}x_{i}$-$H$-path in $D$. For each $i$ in $\{0, 1,  \ldots , n-1\}$ let $T_{i}$ be an $x_{i}x_{i+1}$-$H$-path in $D$ (indices modulo $n$), then we get from Lemma~\ref{lema:unionhtray} that there exists $l$ in $\{1,  \ldots , k\}$ such that $\bigcup \limits _{m=0}^{n-1}T_{m}$ is contained in $G_{l}$, which implies that for each $i$ in $\{0, 1, \ldots , n-1\}$ we have that $T_{i}$ is contained in $G_{l}$. Therefore, $(x_{0}, x_{1}, x_{2}, \ldots , x_{n-1}, x_{n} = x_0, x_{1}, x_{2}, \ldots , x_{n-1}, \ldots)$ is a sequence of vertices that contradicts Lemma \ref{lema:existenciaHseminucleogr}(1).\\
\end{proof}

\begin{lemma}\label{lema:existenciaHseminucleoD}
	Let $H$ be a digraph and $D$ an $H$-colored digraph. Suppose that
	\begin{enumerate}
		\item For every cycle $\gamma$ in $D$ there exists $i$  in $\{ 1, 2, \ldots , k\}$ such that $\gamma$ is contained in $G_{i}$.
		\item For every $H$-walk $P$ in $D$ there exists $j$  in $\{ 1, 2, \ldots , k\}$ such that $P$ is contained in $G_{j}$.
	\end{enumerate} 
	
	Then there exists $x_{0}$ in $V(D)$ such that $\{x_{0}\}$ is an $H$-semikernel of $D$.
\end{lemma}
\begin{proof}Proceeding by contradiction, suppose that for every $x$ in $V(D)$, $\{x\}$ is not an $H$-semikernel of $D$; that is, for every  $x$ in $V(D)$ there exists $y$ in $V(D)\setminus \{x\}$ such that there exists an $xy$-$H$-path in $D$ and there exists no a $yx$-$H$-path in $D$. Therefore, for every $n$ in $\mathbb{N}$ we have that given $x_n$ in $V(D)$ there exists $x_{n+1}$ in $V(D)\setminus \{x_{n}\}$ such that there exists an $x_{n}x_{n+1}$-$H$-path in $D$ and there exists no an $x_{n+1}x_{n}$-$H$-path in $D$ which implies that $(x_{0}, x_{1}, x_{2}, \ldots)$ is a sequence of vertices of $D$ that contradicts Lemma~\ref{lema:noexistenciasucesion}.
	
\end{proof}

Let  $H$ be a digraph and $D$ an $H$-colored digraph. From now on $\{\xi _{1}, \xi_{2}\}$ will be a partition of $\xi$ and for every $i$ in $\{1,2\}$ $D_{i}$ will  denote the spanning subdigraph of $D$ such that  $A(D_{i})=\{a\in A(D): c(a)\in C_{j}$ for some $C_{j}$ in $\xi_{i}\}$. Notice that for every  $r$ in $\{1,2, \ldots , k\}$ $G_{r}$ is a subdigraph of either $D_{1}$ or $D_{2}$.

Let  $H$ be a digraph, $D$ an $H$-colored digraph and $S$ a subset of $V(D)$. We will say that $S$ is an \emph{ $H$-semikernel modulo  $D_{2}$} of $D$ if
\begin{enumerate}
\item $S$ is an $H$-independent set in $D$. 
\item For $z$ in $V(D)\setminus S$ if there exists a $Sz$-$H$-path contained in $D_{1}$, then there exists a  $zS$-$H$-path  in $D$.
 \end{enumerate}

\begin{lemma}\label{lema:existeciaHseminucleomodD2deD}
Let $H$ be a digraph and $D$ an $H$-colored digraph.  Suppose that
	\begin{enumerate}
		\item For every $i$ in $\{1, 2\}$ and for every cycle $\gamma$ contained in $D_{i}$ there exists $C_m$  in $\xi_{i}$ such that $\gamma$ is contained in $G_{m}$.
		\item For every $i$ in $\{1, 2\}$ and for every $H$-walk $P$ contained in $D_{i}$ there exists $C_{m'}$  in $\xi_{i}$ such that $P$ is contained in $G_{m'}$.
	\end{enumerate}  
	
	Then there exists $x_{0}$ in $V(D)$ such that $\{x_{0}\}$ is an $H$-semikernel modulo $D_{2}$ of $D$.
\end{lemma}
\begin{proof} If  $\xi _{1}=\{C_{r}\}$ for some $r$ in $\{ 1, 2, \ldots , k\}$, then $A(G_{r}) = A(D_{1})$ (by definition of $D_1$). Then, we get from Lemma~\ref{lema:existenciaHseminucleogr} that there exists $x_{0}$ in $V(G_{r})$ such that $\{x_{0}\}$ is an $H$-semikernel of $G_{r}$. Therefore, it follows from the definition of $H$-semikernel modulo $D_{2}$ that $\{x_{0}\}$ is an $H$-semikernel modulo $D_{2}$ of $D$.
	
Suppose that $|\xi _{1}|\geq 2$ and let $H'$ be the digraph induced by $\bigcup \limits _{C_{m} \in \xi_{1}}C_{m}$ in $H$, that is $H'=H[\bigcup \limits _{C_{m} \in \xi_{1}}C_{m}]$. Notice that $D_1$ is an $H'$-colored digraph and $\xi_1$ is a partition of $V(H')$ such that for every $C_i$ in $\xi_1$ it holds that $G_{i}=D_{1}[\{a\in A(D_1) : c(a)\in C_{i}\}]=D[\{a\in A(D) : c(a)\in C_{i}\}]$  is transitive by $H'$-paths in $D_1$. Therefore, we get from hypotheses (1) and (2) and Lemma~\ref{lema:existenciaHseminucleoD} that there exists $x_{0}$  in $V(D_{1})=V(D)$ such that $\{x_{0}\}$ is an $H$-semikernel of $D_{1}$. Thus, it follows from the definition of $H$-semikernel modulo $D_{2}$ that $\{x_{0}\}$ is an $H$-semikernel modulo $D_{2}$ of $D$.
\end{proof}

Let $\mathcal{S} =\{S\subseteq V(D): S$ is a nonempty $H$-semikernel  modulo $D_{2}$ of $D\}$.

When $\mathcal{S} \neq \emptyset$ we define the digraph $D_{\mathcal{S}}$ as follows: $V(D_{\mathcal{S}})=\mathcal{S}$ and  for $S_{1}$ and $S_{2}$ in $\mathcal{S}$, with $S_{1}\neq S_{2}$, we have that $(S_{1}, S_{2})\in A(D_{\mathcal{S}})$ if and only if for every $s_{1}$ in $S_{1}$ there exists $s_{2}$ in $S_{2}$ such that either $s_{1}=s_{2}$ or there exists a $s_{1}s_{2}$-$H$-path contained in $D_{2}$ and there exists no a $s_{2}s_{1}$-$H$-path contained in $D$.

\begin{lemma}\label{lema:Dvarsigmaaciclica}
	Let $H$ be a digraph and $D$ an $H$-colored digraph.  Suppose that
	\begin{enumerate}
		\item For every $i$ in $\{1, 2\}$ and for every cycle $\gamma$ contained in $D_{i}$ there exists $C_m$  in $\xi_{i}$ such that $\gamma$ is contained in $G_{m}$.
		\item For every $i$ in $\{1, 2\}$ and for every $H$-walk $P$ contained in $D_{i}$ there exists $C_{m'}$  in $\xi_{i}$ such that $P$ is contained in $G_{m'}$.
	\end{enumerate}  
	
Then there exists the digraph $D_{\mathcal{S}}$ and it is is an acyclic digraph.
\end{lemma}
\begin{proof} It follows from Lemma~\ref{lema:existeciaHseminucleomodD2deD} that $D$ has a nonempty $H$-semikernel modulo $D_{2}$, which implies that $\mathcal{S} \neq \emptyset$ and with this  we can consider the digraph $D_{\mathcal{S}}$.
	
	Proceeding by contradiction, suppose that $D_{\mathcal{S}}$ contains a cycle, say $C=(S_{0}, S_{1}, \ldots , S_{n-1}, S_{0})$, with $n\geq 2$. 
	
	\textbf{Claim 1.} There exists $i_{0}$ in $\{0, 1, \ldots , n-1\}$ such that for some $z$ in $S_{i_{0}}$ we have that $z\notin S_{i_{0}+1}$ (indices  modulo $n$).
	
	Proceeding by contradiction, suppose that for each $i$ in $\{0, 1, \ldots , n-1\}$  and for each $z$ in $S_{i}$ we have that $z\in S_{i+1}$, which implies that $S_{0}\subseteq S_{1}\subseteq S_{2}\subseteq \ldots \subseteq S_{n-1}\subseteq S_{0}$ and therefore $S_{i}=S_{j}$ for every subset $\{i, j\}$ of $\{0, 1, \ldots , n-1\}$, with $i\neq j$, which contradicts that the length of $C$ is at least two.
	
	\textbf{Claim 2.} Let $l_{0}$ be an index in $\{0, 1, \ldots , n-1\}$. If for some $z$ in $S_{l_{0}}$ and for some $w$ in $S_{l_{0}+1}$ we have that there exists a $zw$-$H$-path in $D$, then there exists $j_{0}$ in $\{0, 1, \ldots, n-1\}\setminus \{l_{0}\}$ such that $w\in S_{j_{0}}$ and $w\notin S_{j_{0}+1}$ (indices modulo $n$).
	
	Let $z$ and $w$ two vertices as in claim 2. Suppose without loss of generality that $l_{0}=0$. Observe that $w$ $\notin$ $S_{n}=S_{0}$ because there exists a $zw$-$H$-path in $D$, $\{z, w\}$ $\subseteq$ $S_{0}$ and $S_{0}$ is an $H$-independent set in $D$. Since $w$ $\in$ $S_{1}$, then we can consider $j_{0}=$max$\{i\in \{1, \ldots , n-1\}:w\in S_{i}\}$. Therefore, $w\in S_{j_{0}}$ and $w\notin S_{j_{0}+1}$ by choice of $j_{0}$ (indices modulo $n$).
	
It follows from Claim 1 that there exist $i_{0}$ in $\{0, 1, \ldots , n-1\}$ and $t_{0}$ in $S_{i_{0}}$ such that $t_{0}\notin S_{i_{0}+1}$. Since $(S_{i_{0}}, S_{i_{0}+1}) \in A(D_{\mathcal{S}})$ we get that there exists $t_{1}$ in $S_{i_{0}+1}$ such that there is a $t_{0}t_{1}$-$H$-path contained in $D_{2}$ and there is no a $t_{1}t_{0}$-$H$-path contained in $D$. It follows from Claim 2 that there exists $i_{1}$ in $\{0, 1, \ldots , n-1\}$ such that $t_{1}\in S_{i_{1}}$ and $t_{1}\notin S_{i_{1}+1}$. Since $(S_{i_{1}}, S_{i_{1}+1})\in A(D_{\mathcal{S}})$ we get that there exists $t_{2}$ in $S_{i_{1}+1}$ such that there is a $t_{1}t_{2}$-$H$-path contained in $D_{2}$ and there is no a $t_{2}t_{1}$-$H$-path contained in $D$.

Once chosen $t_{0}, t_{1}, \ldots , t_{m}$ we get from Claim 2 that there exists an index $i_{m}$ in $\{0, 1, \ldots , n-1\}$ such that $t_{m}\in S_{i_{m}}$ and $t_{m}\notin S_{i_{m}+1}$.  Since $(S_{i_{m}}, S_{i_{m}+1})\in A(D_{\mathcal{S}})$ it follows that there exists $t_{m+1}$ in $S_{i_{m}+1}$ such that there is a $t_{m}t_{m+1}$-$H$-path contained in $D_{2}$ and there is no a $t_{m+1}t_{m}$-$H$-path contained in $D$. Thus, we obtain a sequence of vertices $(t_{0}, t_{1}, t_{2}, \ldots)$ such that for every $i$ in $\{0,1, \ldots\}$ there exists a $t_{i}t_{i+1}$-$H$-path contained in $D_{2}$ and there is no a $t_{i+1}t_{i}$-$H$-path in $D$, a contradiction with Lemma~\ref{lema:noexistenciasucesion}, when $|\xi _{2}|\geq 2$. If $|\xi _{2}|=1$, suppose that $\xi _{2}=\{C_{r}\}$, then $A(G_{r}) = A(D_{2})$ and  $(t_{0}, t_{1}, t_{2}, \ldots)$ is a sequence of vertices of $G_{r}$ such that for each $i$ in $\{0,1, \ldots\}$ there exists a $t_{i}t_{i+1}$-$H$-path contained in $G_{r}$ and there exists no a $t_{i+1}t_{i}$-$H$-path in $D$, a contradiction with Lemma~\ref{lema:existenciaHseminucleogr}.
	
Therefore, $D_{\mathcal{S}}$ is an acyclic digraph. 
\end{proof}

\begin{lemma}
\label{lema:cadahcaminoenD1oD2}
Let $H$ be a digraph and $D$ an $H$-colored digraph.  Suppose that
	\begin{enumerate}
		\item  For each $i$ in $\{1, 2\}$ and for each $H$-walk $P$  contained in $D_{i}$ there exists $m'$ in $\{ 1, \ldots , k\}$ such that $P$ is contained in $G_{m'}$.
		\item If either there exists a $\xi_{1}\xi_{2}$-arc or there exists a $\xi_{2}\xi_{1}$-arc in $A(\mathscr{C}_{C}(D))$, say $(a,b)$, then $(a,b) \notin A(H)$.
	\end{enumerate}
	
Then every $H$-walk of $D$ is contained in either $D_{1}$ or in $D_{2}$. Moreover, for each $H$-walk $T$ of $D$ there exists $l$ in $\{1, \ldots , k\}$ such that $T$ is contained in $G_{l}$.
\end{lemma}
\begin{proof}
Let $T=(v_{0},v_{1}, \ldots ,v_{n})$ be an $H$-walk in $D$. Proceeding by contradiction, suppose that $T$ is not contained in neither $D_{1}$ nor $D_{2}$. 
	
If $(v_{0},v_{1}) \in A(D_{1})$, then consider $j=$min$\{i\in\{0,1, \ldots ,n-1\}:(v_{i},v_{i+1})\in A(D_{2})\}$, $\{i\in\{0,1, \ldots ,n-1\}:(v_{i},v_{i+1})\in A(D_{2})\}\neq \emptyset$ because $T$ is not contained in $D_{1}$. It follows from the choice of $j$ that $(v_{j-1},v_{j})\in A(D_{1})$, which implies that $(c(v_{j-1},v_{j}),c(v_{j},v_{j+1}))$ is a $\xi _{1}\xi_{2}$-arc in $A(\mathscr{C}_{C}(D))$. On the other hand, since $T$ is an $H$-walk we get that $(c(v_{j-1},v_{j}),c(v_{j},v_{j+1}))$ $\in$ $A(H)$ which contradicts hypothesis (2).
	
If $(v_{0},v_{1})\in A(D_{2})$, then consider $j=$min$\{i\in\{0,1, \ldots ,n-1\}:(v_{i},v_{i+1})\in A(D_{1})\}$. We get from the choice of $j$ that $(v_{j-1},v_{j})\in A(D_{2})$, which implies that $(c(v_{j-1},v_{j}),c(v_{j},v_{j+1}))$ is a $\xi _{2}\xi_{1}$-arc in $A(\mathscr{C}_{C}(D))$, a contradiction with hypothesis (2) because $T$ is an $H$-walk.

	Therefore, every $H$-walk of $D$ is contained in either $D_{1}$ or in $D_{2}$. On the other hand, it follows from hypothesis (1) that for each $H$-walk $T$ of $D$ there exists $l$ in $\{1, \ldots , k\}$ such that $T$ is contained in $G_{l}$.
\end{proof}

\begin{lemma}\label{lema:existenciaHsubdivision}
	
Let $H$ be a digraph, $D$ an $H$-colored digraph and \{$u$, $z$, $x$, $w$\} a subset of V($D$).  Suppose that
	\begin{enumerate}
		\item For every $i$ in $\{1, 2\}$ and for every cycle $\gamma$ contained in $D_{i}$ there exists $C_m$  in $\xi_{i}$ such that $\gamma$ is contained in $G_{m}$.
		\item For every $i$ in $\{1, 2\}$ and for every $H$-walk $P$ contained in $D_{i}$ there exists $C_{m'}$  in $\xi_{i}$ such that $P$ is contained in $G_{m'}$.
		\item There exists a $uz$-$H$-path contained in $D_{1}$, say $\alpha_{1}$,  there exists a $zw$-$H$-path contained in $D$, say $\alpha_{2}$, and there exists a $wx$-$H$-path contained in $D_{2}$,  say $\alpha_{3}$, where there are $H$-obstructions on $z$ and $w$ with respect to  $\alpha _{1}\cup \alpha _{2}\cup \alpha_{3}$ ($u$ can be $x$).
         \item If either there exists a $\xi_{1}\xi_{2}$-arc or there exists a $\xi_{2}\xi_{1}$-arc in $A(\mathscr{C}_{C}(D))$, say $(a,b)$, then $(a,b) \notin A(H)$.
	\end{enumerate}

If there exist no $uw$-$H$-paths in $D$, there exist no $zx$-$H$-paths in $D$, there exist no $zu$-$H$-paths in $D$, then either there exists a $ux$-path which is a $(\xi _{1}, \xi, \xi_{2})$-$H$-subdivision of $\overrightarrow{P_{3}}$ or there exists a $(\xi_{1}, \xi, \xi_{2})$-$H$-subdivision of $\overrightarrow{C_{3}}$.
\end{lemma}
\begin{proof}

Consider the following remarks which will be useful in the proof of Lemma \ref{lema:existenciaHsubdivision}.
	
	{\bf Remark 1.}
	\begin{enumerate}
		\item[(1)] $u \notin V(\alpha_{2})$, otherwise $(u, \alpha_{2}, w)$ is a $uw$-$H$-path, a contradiction.
		\item[(2)]  $z \notin V(\alpha_{3})$, otherwise$(z,\alpha_{3}, x)$ is a $zx$-$H$-path, a contradiction.
		\item[(3)]  $w \notin V(\alpha_{1})$, otherwise $(u, \alpha_{1}, w)$ is a $uw$-$H$-path, a contradiction.
		\item[(4)]   $x \notin V(\alpha_{2})$, otherwise $(z, \alpha_{2}, x)$ is a $zx$-$H$-path, a contradiction.
		\item[(5)]  If $\beta_{1}$ is an $ab$-$H$-walk contained in $D_{i}$ and  $\beta_{2}$ is a $bc$-$H$-walk contained in $D_{j}$, with $\{i,j\}\subseteq \{1,2\}$, $i\neq j$, then there is an $H$-obstruction on $b$, with respect to $\beta_{1}\cup \beta_{2}$, otherwise $\beta_{1}\cup \beta_{2}$ is an $H$-walk which is no contained in either $D_{1}$ or in $D_{2}$, a contradiction with Lemma~\ref{lema:cadahcaminoenD1oD2}.
	\end{enumerate}
	
	{\bf Remark 2.}  $u$, $z$ and $w$ are three different vertices and $x$ $\notin$ \{$z$, $w$\}.
	
	 It follows from Remark 1.
	
	By Lemma~\ref{lema:cadahcaminoenD1oD2}, we will consider two cases on $\alpha_{2}$.
	
	\textbf{Case 1.} $\alpha_{2}$ is contained in $D_{1}$.

	\textbf{Claim 1.} $V(\alpha_{1}) \cap V(\alpha_{2}) = \{z\}$.
	
	Proceeding by contradiction, suppose that $((V(\alpha_{1})\cap V(\alpha_{2})) - \{z\}) \neq \emptyset$. Then $\alpha_{1} \cup \alpha_{2} $ contains a cycle, say $\gamma$, which has arcs of both  $\alpha_{1}$ and $\alpha_{2}$. Since $\alpha_{1}$ and $\alpha_{2}$ are  contained in $D_1$ we get that  $\gamma$ is contained in $D_{1}$. On the other hand, it follows from hypothesis 1 that there exists $j$ in $\{1, 2, \ldots ,k\}$ such that $\gamma$  is contained in $G_{j}$, which implies that $\alpha_{1}$ and $\alpha _{2}$ are contained in $G_{j}$. Since $G_{j}$ is transitive by $H$-paths we get that there exists a $uw$-$H$-path contained in $G_{i}$, a contradiction with the fact that there exists no $uw$-$H$-path. 
	
	\textbf{Subcase 1.1.} $V(\alpha_{2}) \cap V(\alpha_{3}) = \{w\}$.

	 In this subcase we  consider two subcases.
	
	\textbf{(1.1.)} $V(\alpha_{1}) \cap V(\alpha_{3}) = \emptyset$. 
	
	Notice that by  supposition of this subcase we get  $u\neq x$. Then it follows from Claim 1, the suppossition of the Subcase 1.1 and by hypothesis 3 that  $\alpha_{1}\cup \alpha_{2}\cup\alpha_{3}$ is a $ux$-path which is a $(\xi_{1}, \xi, \xi_{2})$-$H$-subdivision of $\overrightarrow{P_{3}}$.

	\textbf{(1.2)} $V(\alpha_{1}) \cap V(\alpha_{3}) \neq \emptyset$.
	
	Let $y$ be the last vertex in $\alpha_{1}$ which appears in $\alpha_{3}$. Notice that from Remark 1(2) and  Remark 1(3) we get that $y \notin \{z, w\}$.  Since $(y, \alpha_{1},z)$ is contained in $D_1$, $\alpha_{2}$ is contained in $D$, $(w, \alpha_{3}, y)$ is contained in $D_2$, by hypothesis 3 there exist $H$-obstructions on $z$ and $w$, with respect to $(y, \alpha_{1},z) \cup \alpha_{2} \cup (w, \alpha_{3}, y)$,  by Remark 1(5) there exists an $H$-obstruction on $y$, with respect to $(y, \alpha_{1},z) \cup \alpha_{2} \cup (w, \alpha_{3}, y)$,  then it follows from Claim 1, the supposition of Subcase 1.1 and by choice of $y$ that $(y, \alpha_{1},z) \cup \alpha_{2} \cup (w, \alpha_{3}, y)$ is a $(\xi_{1}, \xi, \xi_{2})$-$H$-subdivision of $\overrightarrow{C_{3}}$.

\textbf{Subcase 1.2.} $(V(\alpha_{2}) \cap V(\alpha_{3}) - \{w\}) \neq \emptyset$. 

In this subcase we consider two subcases. 
	
	\textbf{(I.3)} $V(\alpha_{1}) \cap V(\alpha_{3}) = \emptyset$.
	
Let $y$ be the first vertex in $\alpha_{2}$ which appears in $\alpha_{3}$. Notice that it follows from Remark 1(2) and Remark 1(4) that $y \notin \{z, x\}$; and by  supposition $V(\alpha_{1}) \cap V(\alpha_{3}) = \emptyset$ we get that $u\neq x$. Since $\alpha_1$ is contained in $D_1$, $(z, \alpha_{2}, y)$ is contained in $D$,  $(y, \alpha_{3}, x)$ is contained in $D_2$,  by hypothesis 3 there exists an $H$-obstruction on $z$, with respect to $\alpha_1 \cup (z, \alpha_{2},y)\cup (y, \alpha_{3}, x)$, by Remark 1(5) there exists an $H$-obstruction on $y$, with respect to $\alpha_1 \cup (z, \alpha_{2},y )\cup (y, \alpha_{3}, x)$, then it follows from Claim 1 and the choice of $y$ that $\alpha_1 \cup (z, \alpha_{2},y )\cup (y, \alpha_{3}, x)$ is an $ux$-path which is a $(\xi_{1}, \xi, \xi_{2})$-$H$-subdivision of $\overrightarrow{P_{3}}$.

\textbf{(I.4)} $V(\alpha_{1}) \cap V(\alpha_{3}) \neq \emptyset$.

Let $y$ be the first vertex in $\alpha_{3}$ which appears in $\alpha_{1} \cup \alpha_{2}$ and let $e$ be the last vertex in $\alpha_{3}$ which appears in $\alpha_{1}\cup \alpha_{2}$.

If $y\in V(\alpha_{1})$, then we get from Remark 1(2) and Remark 1(3)  that $y \notin \{z, w\}$. Since $(y, \alpha_{1}, z)$ is contained in $D_1$, $\alpha_{2}$ is contained in $D$, $(w, \alpha_{3}, y)$ is contained in $D_2$,  by hypothesis 3 there exist $H$-obstructions on $z$ and $w$, with respect to $(y, \alpha_{1},z) \cup \alpha_{2} \cup (w, \alpha_{3}, y)$, by Remark 1(5) there exists an $H$-obstruction on $y$, with respect to $(y, \alpha_{1},z) \cup \alpha_{2} \cup (w, \alpha_{3}, y)$, then  $(y, \alpha_{1},z) \cup \alpha_{2} \cup (w, \alpha_{3}, y)$ is a $(\xi_{1}, \xi, \xi_{2})$-$H$-subdivision of $\overrightarrow{C_{3}}$.

If $y\in V(\alpha_{2})$ and $e\in V(\alpha_{1})$, we get that $y \neq e$. Let $a$ be the last vertex in $\alpha_{3}$ which appears in $\alpha_{2}$ (there exists $a$ because $y\in V(\alpha_{2})$), and let $b$ be the first vertex in $(a, \alpha_{3}, x)$ which appears in $\alpha_{1}$, (there exists $b$ because $e\in V(\alpha_{1})$ and $e\in V((a, \alpha_{3},x))$). Notice that it follows from the choice of $a$, Remark 1(1), Remark 1(2) and Remark 1(4) that $a$ $\notin$ \{$u$, $z$, $x$\}, also it follows from the choice of  $b$, Remark 1(2) and Remark 1(3) that $b$ $\notin$ \{$z$, $w$\}. So, we get from Claim 1 and the fact $z$ $\notin$ \{$a$, $b$\} that $a \neq b$.  Since $(b, \alpha_{1} , z)$ is contained in $D_1$, $(z, \alpha_{2}, a)$ is contained in $D$, $(a, \alpha_{3}, b)$ is contained in $D_2$, by hypothesis 3 there exists an $H$-obstruction on $z$, with respect to $(b, \alpha_{1} , z) \cup (z, \alpha_{2}, a) \cup (a, \alpha_{3}, b)$, by Remark 1(5) there exist $H$-obstruction on $a$ and $b$, with respect to $(b, \alpha_{1} , z) \cup (z, \alpha_{2}, a) \cup (a, \alpha_{3}, b)$,  then $(b, \alpha_{1} , z) \cup (z, \alpha_{2}, a) \cup (a, \alpha_{3}, b)$ is a $(\xi_{1}, \xi, \xi_{2})$-$H$-subdivision of $\overrightarrow{C_{3}}$.

If $y\in V(\alpha_{2})$ and $e \in V(\alpha_{2})$, then it follows from Remark 1(2) and Remark 1(4) that  $e \notin \{z, x\}$, which implies by  choice of $e$ that $V(\alpha_{1})\cap V((e,\alpha_{3},x))=\emptyset$. Since $\alpha_{1}$ is contained in $D_1$, $(z, \alpha_{2}, e)$ is contained in $D$, $(e, \alpha_{3}, x)$ is contained in $D_2$,  by hypothesis 3 there exists an $H$-obstruction on $z$, by Remark 1(5) there exists an $H$-obstruction on $e$, then  $\alpha_{1} \cup (z, \alpha_{2}, e) \cup (e, \alpha_{3}, x)$ is a $ux$-path which is a $(\xi_{1}, \xi, \xi_{2})$-$H$-subdivision of $\overrightarrow{P_{3}}$.

	\textbf{Case 2.} $\alpha_{2}$ is contained in $D_{2}$.

	\textbf{Claim 2.} $V(\alpha_{2}) \cap V(\alpha_{3}) =\{w\}$.
	
Proceeding by contradiction, suppose that $(V(\alpha_{2})\cap V(\alpha_{3})-\{w\}) \neq \emptyset$. Then $\alpha_{2} \cup \alpha_{3}$ is a walk in $D_{2}$ which contains a cycle $\gamma$. By hypothesis 1 there exists $C_i$ in $\xi_2$, for some $i$ in $\{1, \ldots, k\}$, such that $\gamma$ is contained in $G_{i}$, that implies that $\alpha _{2}\cup \alpha _{3}$ is contained in $G_{i}$ (by hypothesis 2)  and since $G_{i}$ is transitive by $H$-paths then there exists a $zx$-$H$-path, a contradiction. 
	
\textbf{Subcase 2.1.} $V(\alpha_{1}) \cap V(\alpha_{2}) = \{z\}$.
	
In this subcase consider the following subcases.

	\textbf{(2.1.)} $V(\alpha_{1}) \cap V(\alpha_{3}) = \emptyset$. 
	
	Notice that  $u\neq x$ because $V(\alpha_{1}) \cap V(\alpha_{3}) = \emptyset$. Then $\alpha_{1} \cup \alpha_{2} \cup \alpha_{3}$ is a $ux$-path which is a $(\xi_{1}, \xi, \xi_{2})$-$H$-subdivision of $\overrightarrow{P_{3}}$, because by hypothesis 3 there exist $H$-obstructions on $z$ and $w$ with respect to $\alpha_{1} \cup \alpha_{2} \cup \alpha_{3}$.

	\textbf{(2.2.)} $V(\alpha_{1}) \cap V(\alpha_{3}) \neq \emptyset$. 
	
	Let $y$ be the last vertex in $\alpha_{1}$ which appears in $\alpha_{3}$. We get from Remark 1(2) and Remark 1(3) that $y \notin \{z, w\}$. Since $(y, \alpha_{1}, z)$ is contained in $D_1$, $\alpha_2$ is contained in $D$, $(w, \alpha_{3}, y)$ is contained in $D_2$, by hypothesis 3   there exist $H$-obstructions on $z$ and $w$, with respect to $(y, \alpha_{1}, z) \cup \alpha_{2} \cup (w, \alpha_{3}, y)$, by Remark 1(5) there exists an $H$-obstruction on $y$, with respect to $(y, \alpha_{1}, z) \cup \alpha_{2} \cup (w, \alpha_{3}, y)$, then it follows that $(y, \alpha_{1}, z) \cup \alpha_{2} \cup (w, \alpha_{3}, y)$ is a $(\xi_{1}, \xi, \xi_{2})$-$H$-subdivision of $\overrightarrow{C_{3}}$.
	
\textbf{Subcase 2.2.} $(V(\alpha_{1})\cap V(\alpha_{2}) - \{z\}) \neq \emptyset$.

In this subcase consider the following subcases.

\textbf{(2.3.)} $(V (\alpha_{1})\cap V(\alpha_{3})) = \emptyset$.
	
	Let $y$ be the first vertex in $\alpha_{1}$ which appears in  $\alpha_{2}$. Notice that by the supposition in (2.3) we get $u\neq x$,  by supposition of Subcase 2.2 we get $y\neq z$, by Remark 1(1) and Remark 1(3) we get $y \notin \{u, w\}$.  Since $(u, \alpha_{1},y)$ is contained in $D_1$, $(y, \alpha_{2}, w)$ is contained in $D$, $\alpha_{3}$ is contained in $D_2$, there exists an $H$-obstruction on $w$, with respect to $(u, \alpha_{1},y) \cup (y, \alpha_{2}, w) \cup \alpha_{3}$, and  by Remark 1(5) there exists an $H$-obstruction on $y$, with respect to $(u, \alpha_{1},y) \cup (y, \alpha_{2}, w) \cup \alpha_{3}$, then  $(u, \alpha_{1},y) \cup (y, \alpha_{2}, w) \cup \alpha_{3}$ is a $ux$-path  which is a $(\xi_{1}, \xi, \xi_{2})$-$H$-subdivision of $\overrightarrow{P_{3}}$.

	\textbf{(2.4.)} $V (\alpha_{1})\cap V(\alpha_{3}) \neq \emptyset$. 
	
	Let $y$ be  the last vertex in $\alpha_{1}$ which is in $\alpha_{3}$.  Notice that it follows from Remark 1(2) and Remark 1(3) that $y \notin \{z, w\}$. Let $x'$ be the first vertex in   $(y, \alpha_{1}, z)$ which appears in $\alpha_{2}$ (notice that $x'$ can be $z$). It follows from Remark 1(3) that $x' \neq w $, by Claim 2 and the fact $y\neq w$ we get $x' \neq y $. Since $(y, \alpha_{1}, x')$ is contained in $D_1$, $(x', \alpha_{2}, w)$ is contained in $D$, $(w, \alpha_{3}, y)$  is contained in $D_2$,   by Remark 1(5) there exist $H$-obstructions on $y$ and $x'$, with respect to $(y, \alpha_{1}, x') \cup (x', \alpha_{2}, w) \cup (w, \alpha_{3}, y)$,  and by hypothesis 3 there exists an $H$-obstruction on $w$, with respect to $(y, \alpha_{1}, x') \cup (x', \alpha_{2}, w) \cup (w, \alpha_{3}, y)$, then $(y, \alpha_{1}, x') \cup (x', \alpha_{2}, w) \cup (w, \alpha_{3}, y)$ is a $(\xi_{1}, \xi, \xi_{2})$-$H$-subdivision of $\overrightarrow{C_{3}}$.

\end{proof}

\section{Main result}

\begin{theorem}
\label{teo:principal,existenciaHnucleo}

Let $H$ be a digraph and $D$ an $H$-colored digraph.  Suppose that
	\begin{enumerate}
		\item For every $i$ in $\{1, 2\}$ and for every cycle $\gamma$ contained in $D_{i}$ there exists $C_m$  in $\xi_{i}$ such that $\gamma$ is contained in $G_{m}$.
		\item For every $i$ in $\{1, 2\}$ and for every $H$-walk $P$ contained in $D_{i}$ there exists $C_{m'}$  in $\xi_{i}$ such that $P$ is contained in $G_{m'}$.
         \item If either there exists a $\xi_{1}\xi_{2}$-arc or there exists a $\xi_{2}\xi_{1}$-arc in $A(\mathscr{C}_{C}(D))$, say $(a,b)$, then $(a,b) \notin A(H)$.
         \item $D$ does not contain a $(\xi_{1}, \xi, \xi_{2})$-$H$-subdivision of $\overrightarrow{C_{3}}$.
		\item If there exists a $ux$-path which is a $(\xi_{1}, \xi, \xi_{2})$-$H$-subdivision of $\overrightarrow{P_{3}}$, for some subset $\{u,x\}$ of $V(D)$, then there exists a $ux$-$H$-path in $D$.
	\end{enumerate}  	

	Then $D$ has an $H$-kernel.
\end{theorem}
\begin{proof}
	Since $D_{\mathcal{S}}$ is an acyclic digraph (by Lemma~\ref{lema:Dvarsigmaaciclica}), then $D_{\mathcal{S}}$ contains a vertex $S$ such that $d^{+}_{D_{\mathcal{S}}}(S) = 0$ (by Proposition \ref{prop100}). We will prove that $S$ is an $H$-kernel in $D$.
	
	Since $S \in V(D_{\mathcal{S}})$, it follows that $S$ is an $H$-independent set in $D$. Therefore, it remains to prove that $S$ is an $H$-absorbent set in $D$. 
	
	Proceeding by contradiction, suppose  that $S$ is not an $H$-absorbent set in $D$. Let $X$ be the set defined as $\{z \in V(D)\setminus S : $ there exists no a $zS$-$H$-path in $D\}$. Notice that  by our supposition $X \neq \emptyset$.
	
	\textbf{Claim 1.} There exists $x_{0}$ in $X$ such that for every $z$ in $X\setminus \{x_{0}\}$ if there exists an $x_{0}z$-$H$-path contained in $D_{1}$, then there exists a $zx_{0}$-$H$-path contained in $D_{1}$.

	Proceeding by contradiction, suppose that for every $x$ in $X$ there exists $y$ in $X\setminus \{x\}$ such that there exists an $xy$-$H$-path contained in $D_{1}$ and there exists no a $yx$-$H$-path in $D_{1}$. 
	
	Let $x_{0}$ be a vertex in $X$, then there exists $x_{1}$ in $X\setminus\{x_{0}\}$ such that there exists an $x_{0}x_{1}$-$H$-path contained in $D_{1}$ and there exists no an $x_{1}x_{0}$-$H$-path in $D_{1}$. For $x_1$ there exists $x_{2}$ in $X\setminus \{x_{1}\}$ such that there exists an $x_{1}x_{2}$-$H$-path contained in $D_{1}$ and there exists no an $x_{2}x_{1}$-$H$-path contained in $D_{1}$. With this procedure we obtain a sequence of vertices $(x_{0}, x_{1}, x_{2}, \ldots)$ such that for every $i$ in $\{0,1, \ldots\}$ there exists an $x_{i}x_{i+1}$-$H$-path contained in $D_{1}$ and there exists no an $x_{i+1}x_{i}$-$H$-path contained in $D_{1}$.

On the other hand, notice that $D_{1}$ is an $H[\xi _{1}]$-colored digraph.  If $|\xi _{1}|\geq 2$ we get that $\xi _{1}$ is a partition of V($H[\xi _{1}]$). Therefore, it follows from hypotheses 1 and 2  that $D_{1}$ is an $H[\xi _{1}]$-colored digraph which satisfies the hypotheses of Lemma~\ref{lema:noexistenciasucesion}, a contradiction with the existence of the sequence $(x_{0}, x_{1}, x_{2}, \ldots)$. If $|\xi _{1}|=1$, say  $\xi _{1}=\{C_{r}\}$, then A($D_{1}$) = A($G_{r}$), a contradiction with Lemma~\ref{lema:existenciaHseminucleogr} and the existence of the sequence $(x_{0}, x_{1}, x_{2}, \ldots)$. 

Therefore, there exists $x_{0}$ in $X$ such that for every $z$ in $X\setminus \{x_{0}\}$ if there exists an $x_{0}z$-$H$-path contained in $D_{1}$, then there exists a $zx_{0}$-$H$-path contained in $D_{1}$.
	
	Let $T$ be defined as $\{z \in S : $ there exists no a $zx_{0}$-$H$-path in $D_{2}\}$. 
	
	We get from the definition of $T$ that for every $z$ in $S\setminus T$ there exists a $zx_{0}$-$H$-path in $D_{2}$.

\textbf{Claim 2.} $T\cup \{x_{0}\} $ is an  $H$-independent set in $D$.

Since	$T \subseteq S$ and $S$ is an $H$-independent set in $D$, then $T$ is an $H$-independent set in $D$. By definition of $X$ we have that there exists no $x_{0}T$-$H$-path in $D$. By Lemma \ref{lema:cadahcaminoenD1oD2} we have that every $H$-walk of $D$ is contained in either $D_1$ or $D_2$, then it remains to prove that there exists no $Tx_{0}$-$H$-path contained in $D_1$ and there exists no $Tx_{0}$-$H$-path contained in $D_2$. It follows from the fact $S$ is an $H$-semikernel modulo $D_2$, $T\subseteq S$ and the definition of $X$ that there exists no $Tx_{0}$-$H$-path in $D_1$. By definition of $T$ there exists no $Tx_{0}$-$H$-path in $D_2$. Therefore, $T \cup \{x_{0}\}$ is an $H$-independent set in $D$.
	
	\textbf{Claim 3.} For every $z$ in  $V(D)\setminus (T\cup \{x_{0}\})$ if there exists a $(T \cup \{x_{0}\})z$-$H$-path contained in $D_{1}$, then there exists a $z(T \cup \{x_{0}\})$-$H$-path in $D$.

	Let $z$ be in $V(D)\setminus (T\cup \{x_{0}\})$ such that there exists a $(T\cup \{x_{0}\})z$-$H$-path contained in $D_{1}$.
	
	Proceeding by contradiction, suppose that there exists no an $H$-path from $z$ to $T\cup \{x_{0}\}$ in $D$.
	
Consider two cases.
	
	\textbf{Case 1.} There exists a $Tz$-$H$-path contained in $D_{1}$.
	
	Let $u$ be a vertex in $T$ and $\alpha_{1}$  an $H$-path such that $\alpha_{1}$ is a  $uz$-$H$-path contained in $D_{1}$. Since $T \subseteq S$ and $S$ is an $H$-semikernel modulo $D_{2}$ in $D$, then it follows that there exists $w$ in $S$ such that there is a $zw$-$H$-path contained in $D$, say $\alpha_{2}$, which implies that $z\notin S\cup X$. Since there is no an $H$-path from $z$ to $T\cup \{x_{0}\}$ in $D$, we get that $w \in S \setminus T $. It follows from the definition of $T$ that   there exists a $wx_{0}$-$H$-path in $D_{2}$, say $\alpha_{3}$.

There exists  an $H$-obstruction on $z$, with respect to $\alpha_{1} \cup \alpha_{2}$, otherwise  $\alpha_{1} \cup \alpha_{2}$ is a $uw$-$H$-walk in $D$ and by Lemma~\ref{lema:cadahcaminoenD1oD2} there exists $C_m$ in $\xi_1$ such that $\alpha_{1} \cup \alpha_{2}$ is contained in $G_{m}$. Since $G_{m}$ is transitive by $H$-paths we get that there exists a $uw$-$H$-path in $G_{m}$, a contradiction  because $S$ is an $H$-independent set in $D$ and $\{u,w\} \subseteq S$.

There exists an $H$-obstruction on $w$, with respect to $\alpha_{2} \cup \alpha_{3}$, otherwise $\alpha_{2} \cup \alpha_{3}$ is a $zx_{0}$-$H$-walk in $D$ and by Lemma~\ref{lema:cadahcaminoenD1oD2} we get that there exists $C_m$ in $\xi_2$ such that $\alpha_{2} \cup \alpha_{3}$ is contained in $G_{m}$. Since $G_{m}$ is transitive by $H$-paths we get that there exists a  $zx_{0}$-$H$-path, a contradiction because we are supposing that there exists no a $z(T\cup \{x_{0}\})$-$H$-path. Therefore, there exist $H$-obstructions on $z$ and $w$, with respect to  $\alpha_{1}\cup \alpha_{2}\cup \alpha_{3}$. 
	
Notice that there exists no a $uw$-$H$-path in $D$ because $\{u,w\} \subseteq S$ and $S$ is an $H$-independent set in $D$. There exists no a $zx_{0}$-$H$-path and there exists no a $zu$-$H$-path in $D$ because there exists no $z(T\cup \{x_{0}\})$-$H$-paths in $D$.

Therefore, since it holds the hypotheses of Lemma~\ref{lema:existenciaHsubdivision}, we get from hypothesis 4 of Theorem \ref{teo:principal,existenciaHnucleo} that there exists a $ux_{0}$-path which is a $(\xi_{1}, \xi, \xi_{2})$-$H$-subdivision of $\overrightarrow{P_{3}}$, which implies by hypothesis 5 that there exists a $ux_0$-$H$-path in $D$, a contradiction because $T \cup \{x_{0}\}$ is an $H$-independent set in $D$ and $\{u, x_0\} \subseteq (T \cup \{x_{0}\})$.

	\textbf{Case 2.} There exists an $x_{0}z$-$H$-path contained in $D_{1}$. 
	
	Let $\alpha_{1}$ be an $x_{0}z$-$H$-path contained in $D_{1}$. Since there exists no a $z(T\cup \{x_{0}\})$-$H$-path in $D$, then it follows from the choice of $x_0$ that $z \notin X$; in addition, by definition of $X$ we have that $z \notin S$. It follows from the definition of $X$ that there exist $w$ in $S$ and an $H$-path, say $\alpha_{2}$,  such that $\alpha_{2}$ is a $zS$-$H$-path contained in $D$. Notice that $w \notin T$ because we are supposing that there exists no a $z(T\cup \{x_{0}\})$-$H$-path in $D$. It follows from the definition of $T$ that there exists a $wx_{0}$-$H$-path contained in $D_{2}$, say $\alpha_{3}$.
	
There exists an $H$-obstruction on $z$, with respect to  $\alpha_{1} \cup \alpha_{2}$,  otherwise $\alpha_{1} \cup \alpha_{2}$ is an $x_{0}w$-$H$-walk in $D$ and by Lemma~\ref{lema:cadahcaminoenD1oD2} there exists $C_m$ in $\xi_1$ such that $\alpha_{1} \cup \alpha_{2}$ is contained in $G_{m}$. Since $G_{m}$ is transitive by $H$-paths we get that there exists an $x_{0}w$-$H$-path in $D$, a contradiction with the definition of $X$ because  $w\in S$.

There exists an $H$-obstruction on  $w$, with respect to $\alpha_{2} \cup \alpha_{3}$, otherwise $\alpha_{2} \cup \alpha_{3}$ is a $zx_{0}$-$H$-walk in $D$ and by Lemma~\ref{lema:cadahcaminoenD1oD2} there exists $C_m$ in $\xi_2$ such that $\alpha_{2} \cup \alpha_{3}$ is contained in $G_{m}$.  Since $G_{m}$ is transitive by $H$-paths we have that there exists a $zx_{0}$-$H$-path in $D$, a contradiction  because we are supposing that there exists no a $z(T\cup \{x_{0}\})$-$H$-path in $D$.

Denote  $x_{0}$ by $u$. Notice that by definition of $X$ we get that there exists no a $uw$-$H$-path in $D$. Since we are supposing that there exists no a $z(T\cup \{x_{0}\})$-$H$-path in $D$ it follows that there exists no a $zu$-$H$-path in $D$; that is, there exists no a $zx_{0}$-$H$-path.

	Since it holds the hypotheses of Lemma~\ref{lema:existenciaHsubdivision}, we get from hypothesis 4  of Theorem \ref{teo:principal,existenciaHnucleo} that there exists a $ux_{0}$-path which is a $(\xi_{1}, \xi, \xi_{2})$-$H$-subdivision of $\overrightarrow{P_{3}}$, which is not possible because $x_0 = u$.

	Since we obtain a contradiction with cases 1 and 2, we conclude that there exists  an $H$-path from $z$ to $T\cup \{x_{0}\}$ in $D$.

	It follows from Claims 1 and 2 that  $T\cup \{x_{0}\}$ is an $H$-semikernel modulo $D_{2}$ in $D$, which implies that $T\cup \{x_{0}\}$ $\in$ $\mathcal{S}=V(D_{\mathcal{S}})$.
	
	Since $T\subseteq S$  and for every $s$ in $S \setminus T$ there exists a $sx_{0}$-$H$-path contained in $D_{2}$ and there exists no a $x_{0}s$-$H$-path contained in $D$. Then, $(S,T\cup \{x_{0}\})\in A(D_{\mathcal{S}})$, a contradiction with $d ^{+}_{D_{\mathcal{S}}}(S)= 0$.
	
	Therefore, $S$ is an $H$-kernel in $D$.
\end{proof}

\section{Some consequences of Theorem \ref{teo:principal,existenciaHnucleo}}

\begin{corollary}
Let $D$ be a 3-transitive digraph, $H_1$ and $H_2$ two spanning subdigraphs of $D$ such that $A(H_1) \cap A(H_2) = \emptyset$ and $A(H_1) \cup A(H_2) = A(D)$.  Suppose that  $D$ has no $\overrightarrow{C_{3}}$ and for every $i$ in \{1, 2\} $H_i$ is an acyclic digraph. Then $D$ has a kernel.
\end{corollary}
\begin{proof}

It is easy to see that if $\mid A(H_i) \mid = \emptyset$ for some $i$ in \{1, 2\}, then $D$ has a kernel (because in this case by hypothesis $D$ is an acyclic digraph and by Theorem \ref{0000}). Therefore, suppose that $\mid A(H_i) \mid \geq 1$ for every $i$ in \{1, 2\} and $\mid A(D) \mid = q$. 

Let $H$ be a digraph with V($H$) =  \{1, $\ldots$ , $q$\},   and A($H$) = $\emptyset$. Let $D'$ be  the $H$-colored digraph obtained from $D$ by assigning a different color to each arc of $D$. 

Let $\xi=\{\{1\}, \ldots , \{q\}\}$ a partition of V($H$) and \{$\xi_1$ = \{\{$i$\} : $a$ has color $i$ for some $a$ in A($D_1$)\}, $\xi_2$ =  \{\{$j$\} : $b$ has color $j$ for some $b$ in A($D_2$)\}\} a partition of $\xi$.

Notice that it follows from the definition of $\xi$ and the definition of \{$\xi_1$, $\xi_2$\} that for every $i$ in \{1, 2\} $D_i = H_i$. Consider the following claims.

{\bf Claim 1.} For every $i$ in $\{1, 2\}$ and for every cycle $\gamma$ contained in $D_{i}$ there exists $C_m$  in $\xi_{i}$ such that $\gamma$ is contained in $G_{m}$.

This follows because $H_{i}$ is acyclic.

{\bf Claim 2.}  For every $i$ in $\{1, 2\}$ and for every $H$-walk $P$ contained in $D_{i}$ there exists $C_{m'}$  in $\xi_{i}$ such that $P$ is contained in $G_{m'}$.

Since every $H$-path in $D'$ has length at most one, then it  follows from the definition of $\xi$ and the definition of \{$\xi_1$, $\xi_2$\} that Claim 2 holds.

{\bf Claim 3.}  If either there exists a $\xi_{1}\xi_{2}$-arc or there exists a $\xi_{2}\xi_{1}$-arc in $A(\mathscr{C}_{C}(D'))$, say $(a,b)$, then $(a,b) \notin A(H)$.

It follows from the fact $A(H) = \emptyset$.

{\bf Claim 4.}  $D'$ does not contain a $(\xi_{1}, \xi, \xi_{2})$-$H$-subdivision of $\overrightarrow{C_{3}}$.

It follows from the fact $D$ has no $\overrightarrow{C_{3}}$ and because every $H$-path in $D'$ has length at most one.

{\bf Claim 5.}  If there exists a $ux$-path which is a $(\xi_{1}, \xi, \xi_{2})$-$H$-subdivision of $\overrightarrow{P_{3}}$, for some subset $\{u,x\}$ of $V(D)$, then there exists a $ux$-$H$-path in $D$.

Since $D$ is a 3-transitive digraph, then Claim 5 holds.\\

Therefore, we get from Theorem \ref{teo:principal,existenciaHnucleo} that $D'$ has an $H$-kernel which is a kernel.
 
\end{proof}

{\bf Proof of Theorem \ref{ccd}}

Consider the partition $\xi$ = \{\{$v$\} : $v$ $\in$ V($H$)\} of $V(H)$ and the partition \{$\xi_1$ = \{\{$u$\} : $u$ $\in$ $X$\}, $\xi_2$ = \{\{$u$\} : $u$ $\in$ $Y$\}\} of $\xi$ which were given in  Remark \ref{rembip}. Consider the following claims.

\begin{enumerate}
		\item For every $i$ in $\{1, 2\}$ and for every cycle $\gamma$ contained in $D_{i}$ there exists $C_m$  in $\xi_{i}$ such that $\gamma$ is contained in $G_{m}$.
		
It follows from the fact \{$X$, $Y$\} is a partition of V($\mathscr{C}_{C}(D)$) into independent sets.

		\item For every $i$ in $\{1, 2\}$ and for every $H$-walk $P$ contained in $D_{i}$ there exists $C_{m'}$  in $\xi_{i}$ such that $P$ is contained in $G_{m'}$.
		
		It follows from the fact \{$X$, $Y$\} is a partition of V($\mathscr{C}_{C}(D)$) into independent sets.
		
         \item If either there exists a $\xi_{1}\xi_{2}$-arc or there exists a $\xi_{2}\xi_{1}$-arc in $A(\mathscr{C}_{C}(D))$, say $(a,b)$, then $(a,b) \notin A(H)$.
         
It follows from the fact $A(H)$ = \{($u$,$u$) : $u$ $\in$ V($H$)\}.

         \item $D$ does not contain a $(\xi_{1}, \xi, \xi_{2})$-$H$-subdivision of $\overrightarrow{C_{3}}$.

Proceeding by contradiction, suppose that $W=(u_{0}, \ldots , u_{l}=v_{0}, \ldots , v_{m}=w_{0}, \ldots ,w_{n}=u_{0})$  is a $(\xi_{1}, \xi, \xi_{2})$-$H$-subdivision of $\overrightarrow{C_{3}}$. Since there are $H$-obstructions on $u_{0}$, $v_{0}$ and $w_{0}$ with respect to $W$, then the color of the monochromatic path $T_{1}=(u_{0},W, u_{l})$ is different of the color of the monochromatic path $T_{2}=(v_{0},W, v_{m})$ and the color of the monochromatic path $T_2$  is different of the color of the monochromatic path $T_{3}=(w_{0}, W, w_{n})$. Suppose that $T_1$ has color $i$, $T_2$ has color $j$ and $T_3$ has color $k$. Since $i \neq j$,  ($i$,$j$) $\in$ A($\mathscr{C}_{C}(D)$), $i$ $\in$ $X$ and $X$ is an independent set in $\mathscr{C}_{C}(D)$, we get that $j$ $\in$ $Y$, which implies that $\{j, k\} \subseteq Y$, a contradiction because $k \neq j$, ($j$,$k$) $\in$ A($\mathscr{C}_{C}(D)$) and $Y$ is an independent set in $\mathscr{C}_{C}(D)$.

\item If there exists a $ux$-path which is a $(\xi_{1}, \xi, \xi_{2})$-$H$-subdivision of $\overrightarrow{P_{3}}$, for some subset $\{u,x\}$ of $V(D)$, then there exists a $ux$-$H$-path in $D$.

The proof is similar that 4.\\

Therefore, we get from Theorem \ref{teo:principal,existenciaHnucleo} that $D$ has an $H$-kernel which is an mp-kernel.$\blacksquare$
\end{enumerate}

\begin{lemma}
\label{prorain}
Let $D$ be an m-colored digraph, \{$u$, $v$, $w$\} a subset of V($D$), with $u\neq v$ $u\neq w$ $w\neq v$,   $P_1$  a $uv$- path in $D$ and $P_2$  a $vw$-path in $D$. Suppose that every chromatic class is transitive. Then 
\begin{enumerate}
\item if   $P_1$  and $P_2$ are properly colored paths, then there exists a $uw$-properly colored path in $D$,
\item  if  $\mathscr{C}_{C}(D)$ has no cycles of length at least two, then every properly colored path is a rainbow path.
\end{enumerate}
\end{lemma}
\begin{proof}
Suppose that $P_1 = (u= u_0, \ldots , u_n=v)$ and $P_2 = (v= v_0, \ldots , v_m=w)$ for some subset \{$n$, $m$\} of $\mathbb{N}$. 

1. Suppose that $P_1$ and $P_2$ are properly colored paths. Let $u_{i_0}$ be the first vertex in $P_1$ that appears in $P_2$;  there exists $u_{i_0}$ because $v$ $\in$ $V(P_1) \cap V(P_2)$. If  either $u_{i_0} = u$ or $u_{i_0} = w$, then clearly ($u_{i_0} = u$, $P_2$, $w$) is a $uw$-properly colored path in $D$ or ($u$, $P_1$, $u_{i_0} = w$) is a $uw$-properly colored path in $D$, respectively; suppose that $u_{i_0} \neq u$ and $u_{i_0} \neq w$. Notice that  ($u$, $P_1$, $u_{i_0}$) and ($u_{i_0} = u$, $P_2$, $w$) are properly colored paths and by choice of $u_{i_0}$ we have that ($u$, $P_1$, $u_{i_0}$) $\cup$ ($u_{i_0}=v_j$, $P_2$, $w$) is a path. If  ($u_{i_0-1}$,$u_{i_0}$) and ($u_{i_0}=v_j$,$v_{j+1}$) have a different color, then ($u$, $P_1$, $u_{i_0}$) $\cup$ ($u_{i_0}=v_j$, $P_2$, $w$) is a $uw$-properly colored path in $D$; therefore suppose that ($u_{i_0-1}$,$u_{i_0}$) and ($u_{i_0}=v_j$,$v_{j+1}$) have the same color, say $\alpha$. Since ($u_{i_0-1}$,$u_{i_0}$) and ($u_{i_0}=v_j$,$v_{j+1}$) are in  the same chromatic class, then it follows from hypothesis  that ($u_{i_0-1}$,$v_{j+1}$) $\in$ A($D$) and this arc has color $\alpha$. Therefore, ($u$, $P_1$, $u_{i_0-1}$) $\cup$ ($u_{i_0-1}$,$v_{j+1}$) $\cup$ ($v_{j+1}$, $P_2$, $w$) is a $uw$-properly colored path in $D$.

2. Let  $P_3 = (u= w_0, \ldots , w_k=w)$ be a properly colored path in $D$  for some $k$ in $\mathbb{N}$. We claim that $P_3$ is a rainbow path, otherwise there exist two arcs ($w_i$,$w_{i+1}$) and ($w_j$,$w_{j+1}$) in $P_3$, with $i+1 < j$, such that these arcs have the same color, say $\alpha$. Therefore, ($c$($w_i$,$w_{i+1}$) = $\alpha$, $c$($w_{i+1}$,$w_{i+2}$), $\ldots$ , $c$($w_j$,$w_{j+1}$) = $\alpha$) is a closed walk in $\mathscr{C}_{C}(D)$ which contains a cycle (by Proposition \ref{closedcycle}), a contradiction.

\end{proof}

\begin{corollary}
\label{altpath}
Let $D$ be an m-colored digraph without isolated vertices. Suppose that every chromatic class is transitive. Then $D$ has a PCP-kernel.
\end{corollary}
\begin{proof}
Let $H$ be a complete digraph without loops such  that  V($H$) = \{$i$ : ($u$,$v$) has color $i$ for some ($u$,$v$) in A($D$)\}. Notice that $D$ is an $H$-colored digraph (by choice of $H$).

Let $D^{*}$ be a digraph such that $D$ and $D^{*}$ are isomorphic, with V($D$) $\cap$ V($D^{*}$) = $\emptyset$, and let $H^{*}$ be a digraph such that $H$ and $H^{*}$ are isomorphic, with V($H$) $\cap$ V($H^{*}$) = $\emptyset$. Consider $f : V(D) \rightarrow V(D^{*})$ and $g : V(H) \rightarrow V(H^{*})$ two isomorphisms. Suppose that $D^{*}$ is an $H^{*}$-colored digraph such that ($u$,$v$) has color $i$ in $D$ if and only if ($f(u)$,$f(v)$) has color $g(i)$ in $D^{*}$. It follows that $D^{*}$ holds the same  hypothesis as $D$.

Let $D' = D \cup D^*$ be. Notice that $D'$ is an $H'$-colored digraph (with V($H'$) = V($H$) $\cup$ V($H^*$) and A($H'$) = A($H$) $\cup$ A($H^*$))

Consider the partition $\xi$ = \{$C_1$ = V($H$), $C_2$ = V($H^{*}$)\} of $V(H')$ and the partition \{$\xi_1$ = \{$C_1$\}, $\xi_2$ = $C_2$\}\} of $\xi$. It follows from the choice of $\xi$ and \{$\xi_1$, $\xi_2$\} that $A(D_1)=A(D)$ and $A(D_2) = A(D^{*})$. Since every $H'$-path in $D'$ is a properly colored path in $D'$, we get from Lemma \ref{prorain} (1) that $G_{i}$ = $D'$[\{$a\in$ A($D'$) : the color of the arc $a$ is in $C_{i}$\}] is a subdigraph of $D'$ which is transitive by $H'$-paths in $D'$ for every $i$ in \{1, 2\}.  The following claims can be deduced from the definition of $D'$ and its $H'$-coloring.

\begin{enumerate}
\item For every $i$ in $\{1, 2\}$ and for every cycle $\gamma$ contained in $D_{i}$ there exists $C_m$  in $\xi_{i}$ such that $\gamma$ is contained in $G_{m}$.

\item For every $i$ in $\{1, 2\}$ and for every $H'$-walk $P$ contained in $D_{i}$ there exists $C_{m'}$  in $\xi_{i}$ such that $P$ is contained in $G_{m'}$.

\item If either there exists a $\xi_{1}\xi_{2}$-arc or there exists a $\xi_{2}\xi_{1}$-arc in $A(\mathscr{C}_{C}(D'))$, say $(a,b)$, then $(a,b) \notin A(H')$.

\item $D'$ does not contain a $(\xi_{1}, \xi, \xi_{2})$-$H'$-subdivision of $\overrightarrow{C_{3}}$.

\item If there exists a $ux$-path which is a $(\xi_{1}, \xi, \xi_{2})$-$H'$-subdivision of $\overrightarrow{P_{3}}$, for some subset $\{u,x\}$ of $V(D')$, then there exists a $ux$-$H'$-path in $D'$.
\end{enumerate}

Therefore, we get from Theorem \ref{teo:principal,existenciaHnucleo} that $D'$ has an $H'$-kernel, say $K$. It follows from the construction of $D'$ that V($D$) $\cap$ $K$ is an $H$-kernel of $D$, which is a PCP-kernel of $D$.
          
\end{proof}

\begin{corollary}
Let $D$ be an m-colored digraph without isolated vertices. Suppose that every chromatic class is transitive. If $\mathscr{C}_{C}(D)$ has no cycles of length at least two, then $D$ has a kernel by rainbow paths.
\end{corollary}
\begin{proof}
Let $N$ be a  PCP-kernel of $D$  (by Corollary \ref{altpath}). It follows from Lemma \ref{prorain} (2) that $N$ is a kernel by rainbow paths (recall that every rainbow path is a properly colored path).
\end{proof}

Recall that when $A(H)$ = \{($u$,$u$) : $u$ $\in$V($H$)\}, an $H$-path is a monochromatic path,  an $H$-kernel is an  mp-kernel and $D$ is said to be $m$-colored (where $|V(H)|=m$). Also in this case a $(\xi_{1}, \xi, \xi_{2})$-$H$-subdivision of $\overrightarrow{C_{3}}$ is called $(\xi_{1}, \xi, \xi_{2})$-subdivision of $\overrightarrow{C_{3}}$ and a $(\xi_{1}, \xi, \xi_{2})$-$H$-subdivision of $\overrightarrow{P_{3}}$ is called $(\xi_{1}, \xi, \xi_{2})$-subdivision of $\overrightarrow{P_{3}}$.

\begin{corollary}[\cite{galeana}]
Let $D$ be an m-colored digraph, $C$  the set of colors of $A(D)$, $\xi$ =\{$C_{1}$, $C_{2}$, $\ldots$ ,  $C_{k}$\} ($k\geq 2$) a partition of $C$, $\{\xi_{1}, \xi_{2}\}$ a partition of $\xi$.   For every $i$ in \{1, 2\} $D_i$ is the spanning subdigraph of $D$ such that $A(D_i)=\{a\in A(D) : c(a)\in C ~\text{for some} ~C ~ \text{in} ~ \xi_i\}$. Suppose that

\begin{enumerate}
\item for every $i$ in \{1, $\ldots$ , $k$\}, $D[\{a\in A(D) : c(a)\in C_{i}\}]$ is transitive by monochromatic paths,
\item for every $i$ in \{1, 2\} and for every cycle $\gamma$ contained in $D_i$ there exists $C_j$ in $\xi_i$ such that $c(e)\in C_j$ for every $e$ in $A(\gamma)$,
\item $D$ does not contain a $(\xi_{1}, \xi, \xi_{2})$-subdivision of $\overrightarrow{C_{3}}$,
\item if there exists a $ux$-path which is a $(\xi_{1}, \xi, \xi_{2})$-subdivision of $\overrightarrow{P_{3}}$, for some subset $\{u,x\}$ of $V(D)$, then there exists a $ux$-monochromatic path in $D$.
\end{enumerate}

Then $D$ has an mp-kernel.
\end{corollary}

The following examples show that each hypothesis in Theorem \ref{teo:principal,existenciaHnucleo} is tight. We  show digraphs $H$ and $D$ such that $D$ is an $H$-colored digraph without $H$-kernel.

\begin{remark}
The condition ``for every $i$ in $\{1, 2\}$ and for every cycle $\gamma$ contained in $D_{i}$ there exists $C_m$  in $\xi_{i}$ such that $\gamma$ is contained in $G_{m}$" in Theorem \ref{teo:principal,existenciaHnucleo} cannot be dropped as example in Figure \ref{fig:ejemplo1} shows.
\end{remark}

\begin{figure}[h!]
\begin{center}
\includegraphics[scale=0.3]{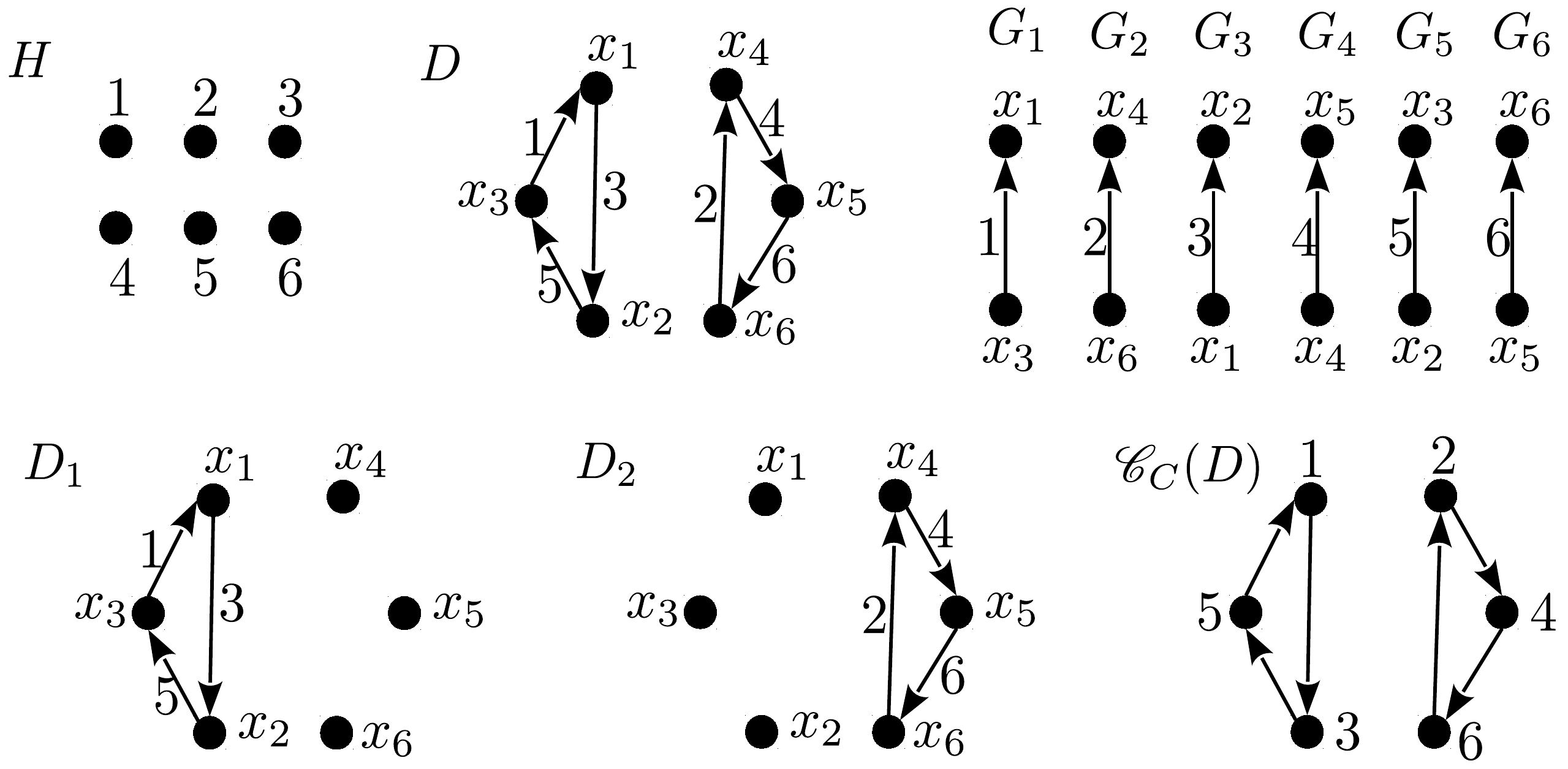}
\caption{$D$ is an $H$-colored digraph without $H$-kernel, $\xi=\{C_{1}=\{1\},C_{2}=\{2\}, C_{3}=\{3\},C_{4}=\{4\}, C_{5}=\{5\}, C_{6}=\{6\}\}$ is a partition of $V(H)$ and $\{\xi_{1}=\{\{1\},\{3\},\{5\}\}, \xi_{2}=\{\{2\},\{4\},\{6\}\}\}$ is a partition of $\xi$.  With these two partitions $D$ holds the hypotheses 2, 3, 4 and 5 of  Theorem \ref{teo:principal,existenciaHnucleo} but $D$ does not hold the hypothesis 1. }
\label{fig:ejemplo1}
\end{center}
\end{figure}

\begin{remark}
The condition ``for every $i$ in $\{1, 2\}$ and for every $H$-walk $P$ contained in $D_{i}$ there exists $C_{m'}$  in $\xi_{i}$ such that $P$ is contained in $G_{m'}$" in Theorem \ref{teo:principal,existenciaHnucleo} cannot be dropped as example in Figure \ref{fig:ejemplo2} shows.
\end{remark}

\begin{figure}[h!]
			\begin{center}
				
					\includegraphics[scale=0.3]{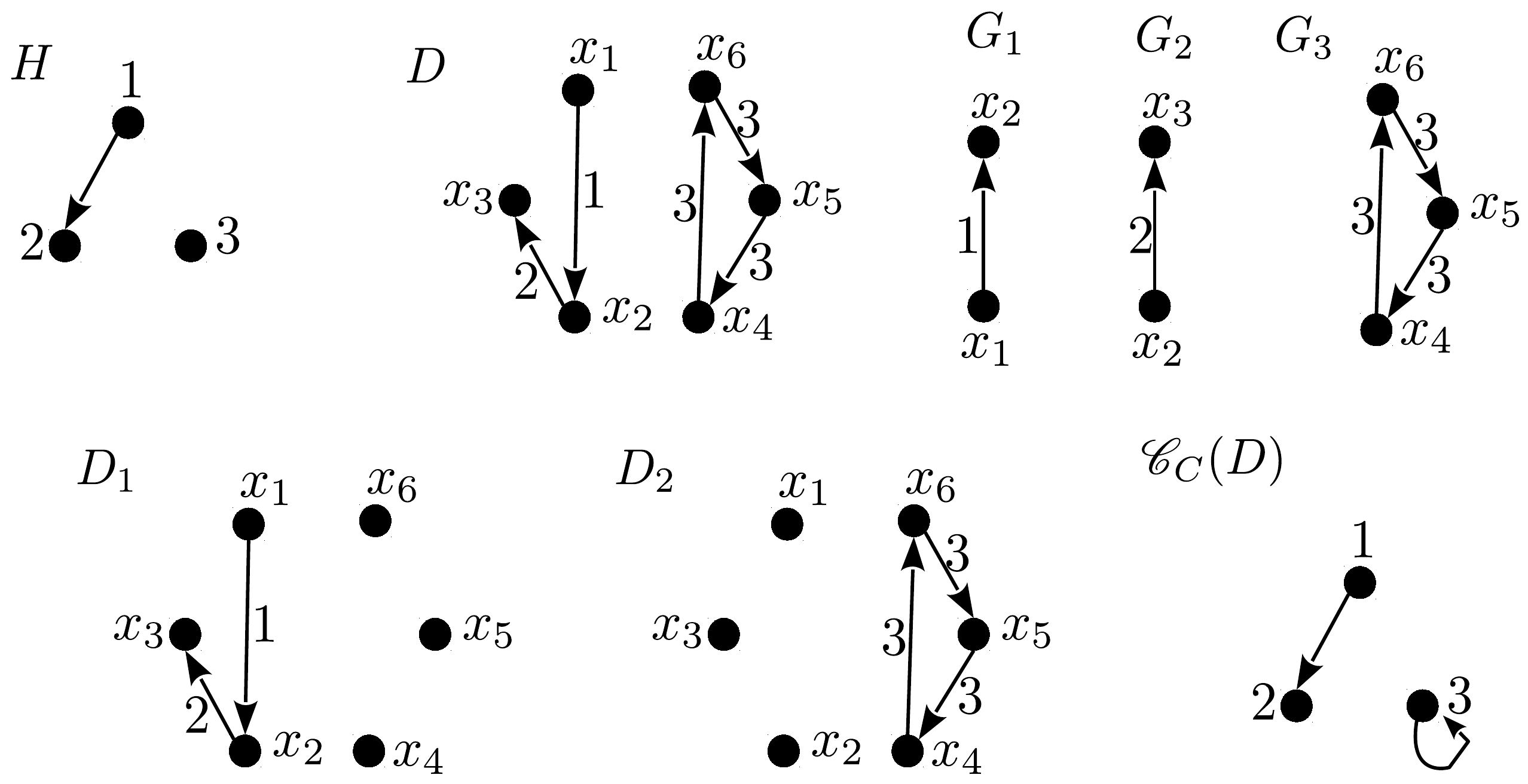}
					
				\caption{$D$ is an $H$-colored digraph without $H$-kernel, $\xi=\{C_{1}=\{1\},C_{2}=\{2\}, C_{3}=\{3\}\}$ is a partition of $V(H)$ and $\{\xi_{1}=\{\{1\},\{2\}\}, \xi_{2}=\{\{3\}\}\}$ is a partition of $\xi$.  With these two partitions $D$ holds the hypotheses 1, 3, 4 and 5 of  Theorem \ref{teo:principal,existenciaHnucleo} but $D$ does not hold  hypothesis 2. }
				\label{fig:ejemplo2}
			\end{center}
		\end{figure}

\newpage

\begin{remark}
The condition ``if either there exists a $\xi_{1}\xi_{2}$-arc or there exists a $\xi_{2}\xi_{1}$-arc in $A(\mathscr{C}_{C}(D))$, say $(a,b)$, then $(a,b) \notin A(H)$" in Theorem \ref{teo:principal,existenciaHnucleo} cannot be dropped as example in Figure \ref{fig:ejemplo3} shows.
\end{remark}

\begin{figure}[h!]
			\begin{center}
					\includegraphics[scale=0.3]{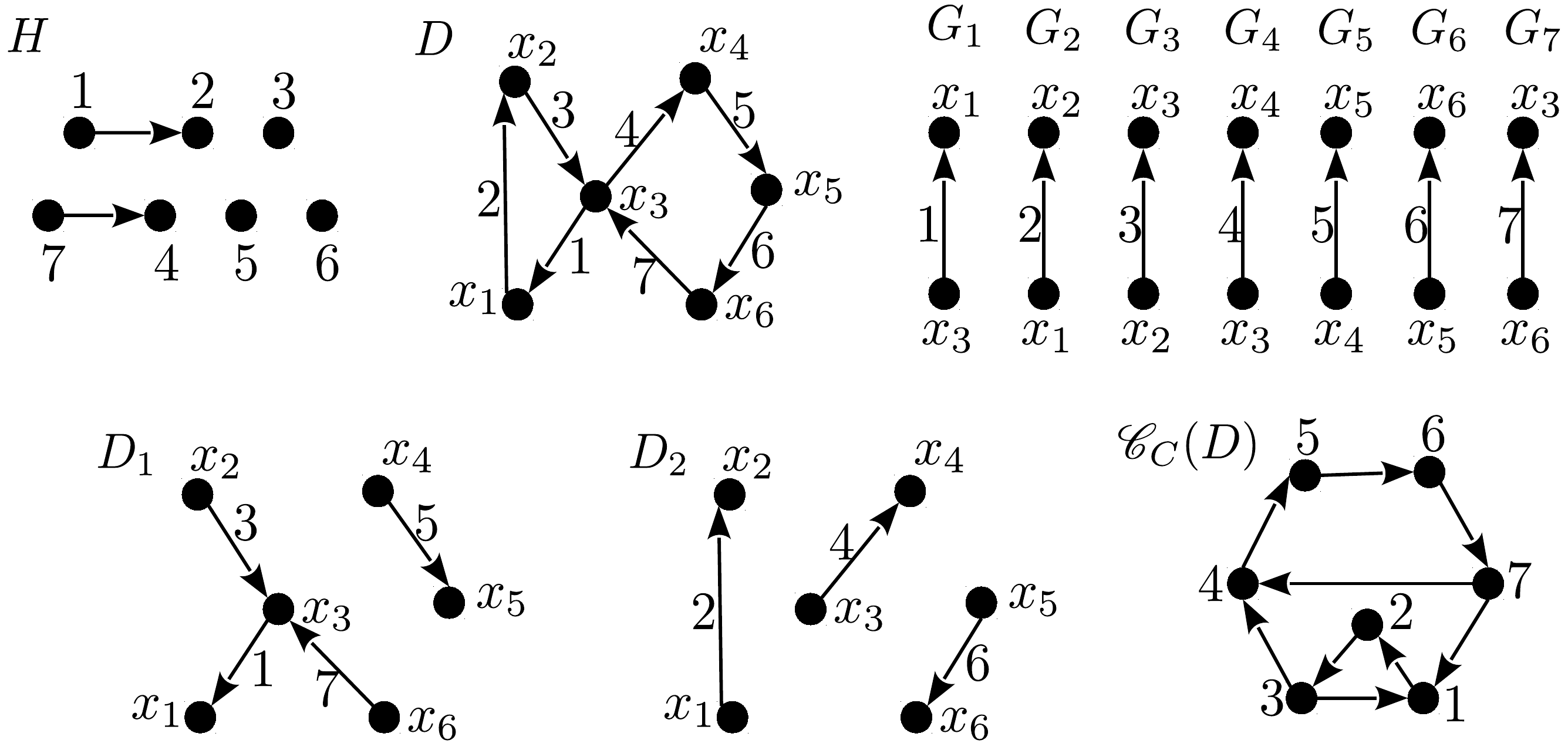}
				\caption{$D$ is an $H$-colored digraph without $H$-kernel, $\xi=\{C_{1}=\{1\},C_{2}=\{2\}, C_{3}=\{3\},C_{4}=\{4\},C_{5}=\{5\}, C_{6}=\{6\}\,C_{7}=\{7\}\}$ is a partition of $V(H)$ and $\{\xi_{1}=\{\{1\},\{3\}, \{5\},\{7\}\}, \xi_{2}=\{\{2\},\{4\},\{6\}\}\}$ is a partition of $\xi$.  With these two partitions $D$ holds the hypotheses 1, 2, 4 and 5 of  Theorem \ref{teo:principal,existenciaHnucleo} but $D$ does not hold hypothesis 3. }
				\label{fig:ejemplo3}
			\end{center}
		\end{figure}

\begin{remark}
The condition ``$D$ does not contain a $(\xi_{1}, \xi, \xi_{2})$-$H$-subdivision of $\overrightarrow{C_{3}}$" in Theorem \ref{teo:principal,existenciaHnucleo} cannot be dropped as example in Figure \ref{fig:ejemplo4} shows.
\end{remark}

\begin{figure}[h!]
			\begin{center}
				\includegraphics[scale=0.3]{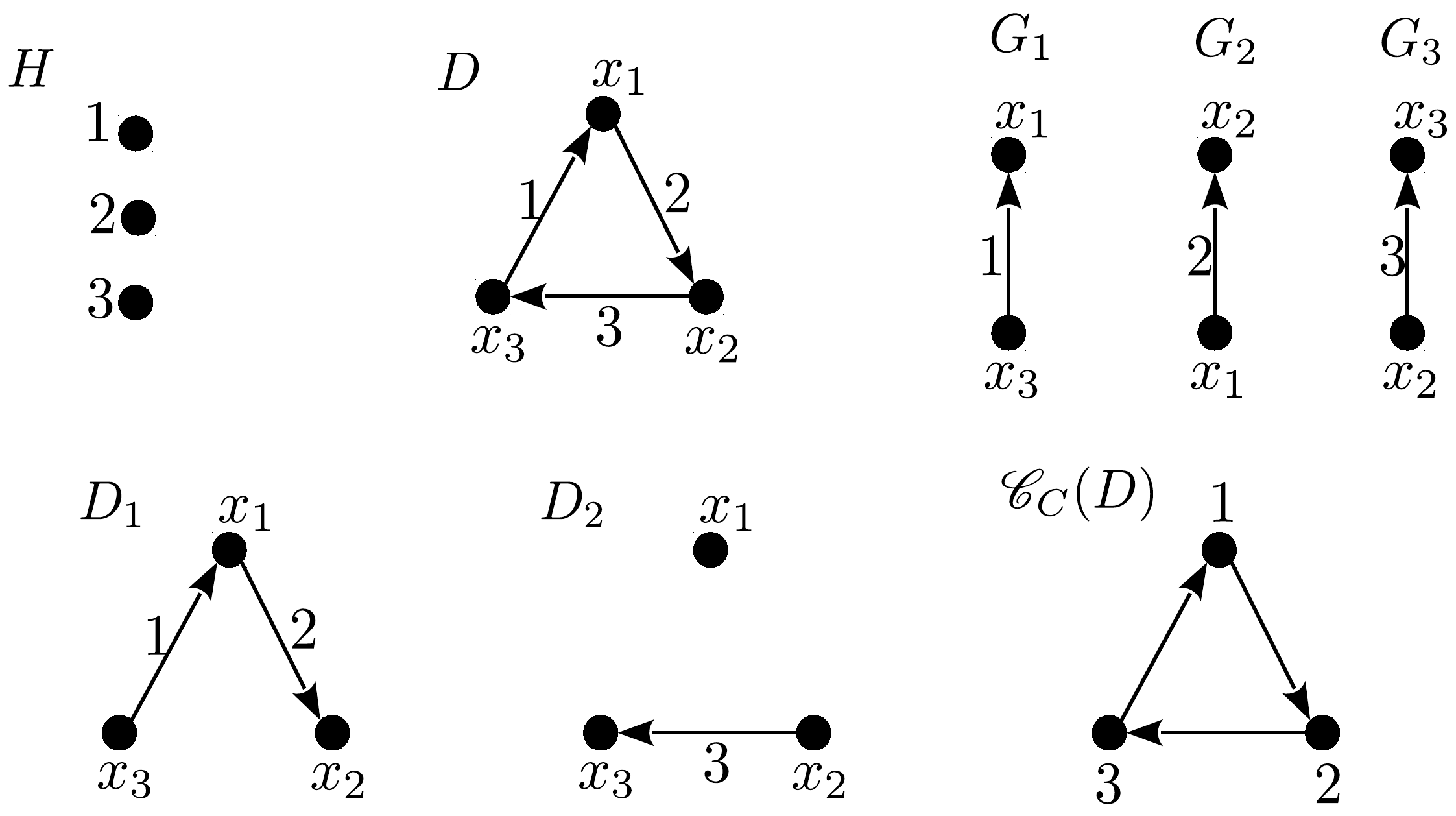}
				\caption{$D$ is an $H$-colored digraph without $H$-kernel, $\xi=\{C_{1}=\{1\},C_{2}=\{2\}, C_{3}=\{3\}\}$ is a partition of $V(H)$ and $\{\xi_{1}=\{\{1\},\{2\}\}, \xi_{2}=\{\{3\}\}\}$ is a partition of $\xi$.  With these two partitions $D$ holds the hypotheses 1, 2, 3 and 5 of  Theorem \ref{teo:principal,existenciaHnucleo} but $D$ does not hold the hypothesis 4.}
				\label{fig:ejemplo4}
			\end{center}
		\end{figure}

\begin{remark}
The condition ``if there exists a $ux$-path which is a $(\xi_{1}, \xi, \xi_{2})$-$H$-subdivision of $\overrightarrow{P_{3}}$, for some subset $\{u,x\}$ of $V(D)$, then there exists a $ux$-$H$-path in $D$" in Theorem \ref{teo:principal,existenciaHnucleo} cannot be dropped as example in Figure \ref{fig:ejemplo5} shows.
\end{remark}

\begin{figure}[h!]
			\begin{center}
				\includegraphics[scale=0.3]{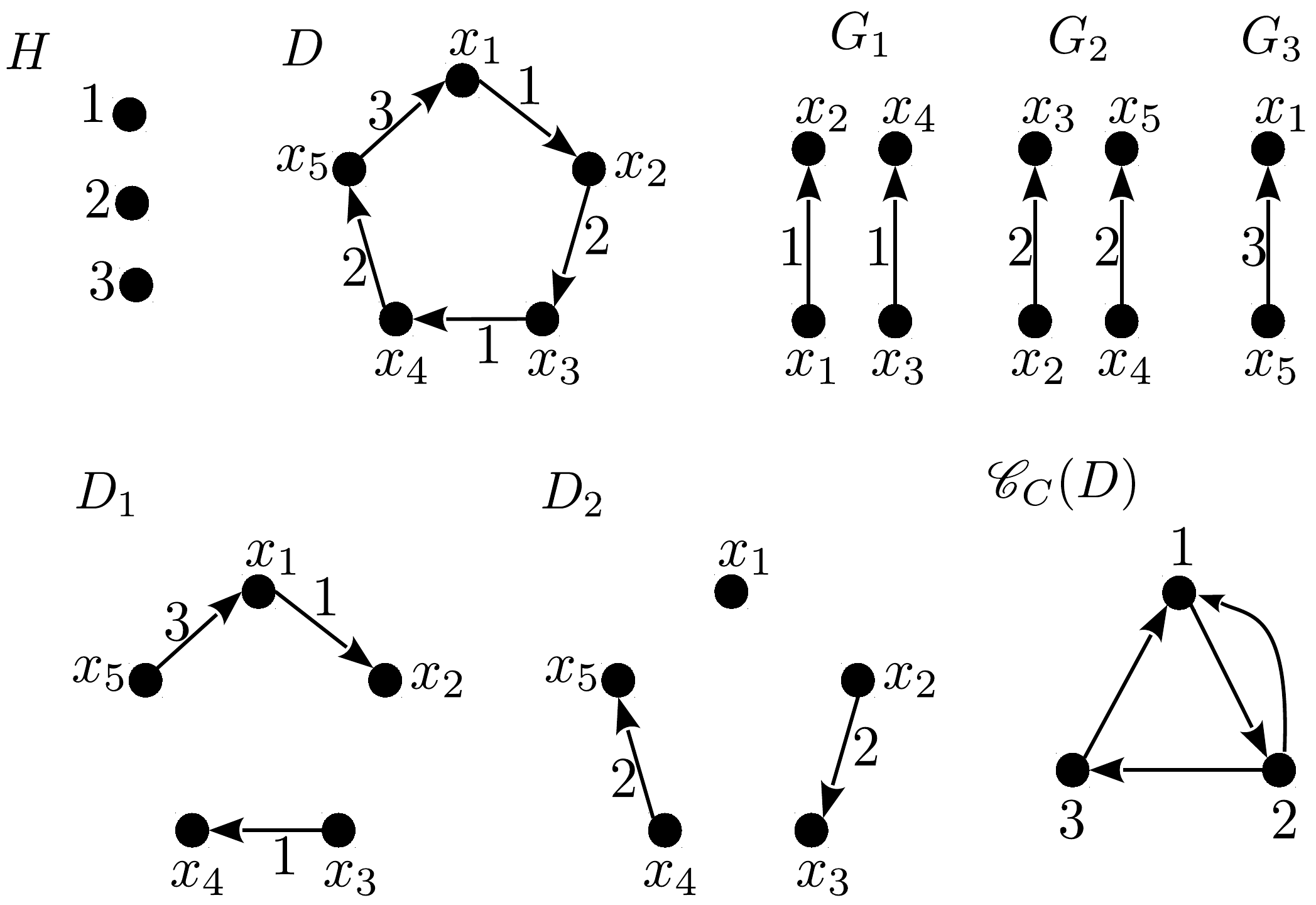}
				\caption{$D$ is an $H$-colored digraph without $H$-kernel, $\xi=\{C_{1}=\{1\},C_{2}=\{2\}, C_{3}=\{3\}\}$ is a partition of $V(H)$ and $\{\xi_{1}=\{\{1\},\{3\}\}, \xi_{2}=\{\{2\}\}\}$ is a partition of $\xi$.  With these two partitions $D$ holds the hypotheses 1, 2, 3 and 4 of  Theorem \ref{teo:principal,existenciaHnucleo} but $D$ does not hold the hypothesis 5.}
				\label{fig:ejemplo5}
			\end{center}
		\end{figure}

\end{document}